\documentclass[a4paper]{amsart}

\usepackage{amsmath,multicol}
\usepackage{amssymb,mathrsfs}
\usepackage{amsthm}
\usepackage[all]{xypic}
\usepackage{amsfonts}
\usepackage{fullpage}
\usepackage{verbatim}
\usepackage{hyperref}

\numberwithin{equation}{section}
\theoremstyle{plain}
\newtheorem{theorem}[equation]{Theorem}
\newtheorem{proposition}[equation]{Proposition}
\newtheorem{corollary}[equation]{Corollary}
\newtheorem{lemma}[equation]{Lemma}

\newtheorem{conjecture}[equation]{Conjecture}

\theoremstyle{definition}
\newtheorem{definition}[equation]{Definition}

\newtheorem{notation}[equation]{Notation}

\newtheorem{example}[equation]{Example}

\newtheorem{remark}[equation]{Remark}

\newcommand{\DMO}{\DeclareMathOperator}

\newcommand{\beq}{\begin{equation}}
\newcommand{\eeq}{\end{equation}}

\newcommand{\bbar}[1]{\overline{#1}}

\DeclareMathOperator{\id}{{Id}}
\DeclareMathOperator{\Hom}{{Hom}}

\DeclareMathOperator{\End}{{End}}
\DeclareMathOperator{\Ext}{{Ext}}

\DeclareMathOperator{\Tor}{Tor}
\DeclareMathOperator{\Aut}{{Aut}}
\DeclareMathOperator{\Proj}{Proj}

\DeclareMathOperator{\pd}{pdim}

\DeclareMathOperator{\coker}{Coker}
\DeclareMathOperator{\im}{Im}

\DeclareMathOperator{\gr}{gr}
\DeclareMathOperator{\Div}{Div}

\DeclareMathOperator{\Ann}{Ann}

\DeclareMathOperator{\grk}{grk}

\newcommand{\mc}{\mathcal}

\newcommand{\kk}{{\Bbbk}}
\newcommand{\timesg}{{\raisebox{1pt}{$\scriptscriptstyle\bullet$}} g}
\newcommand{\bigdot}{{\raisebox{0.5pt}{$\scriptscriptstyle\bullet$}}}
\newcommand{\dotms}{{\,\raisebox{1pt}{${\scriptscriptstyle\bullet}_{_{M\hskip -1pt S}}$}}}

\newcommand{\ZZ}{{\mathbb Z}}

\newcommand{\PP}{{\mathbb P}}

\newcommand{\FF}{{\mathbb F}}
\newcommand{\mb}{\mathbb}
\newcommand{\wt}{\widetilde}

\newcommand{\II}{\mb I}
\newcommand{\JJ}{\mb J}
\newcommand{\KK}{\mb K}
\newcommand{\LL}{\mb L}

\newcommand{\sL}{\mc{L}}
\newcommand{\sM}{\mc{M}}
\newcommand{\sN}{\mc{N}}

\DeclareMathOperator{\Qcoh}{Qcoh}
\DeclareMathOperator{\coh}{coh}
\DeclareMathOperator{\rgr}{gr-\!}
\DeclareMathOperator{\lgr}{\!-gr}

\DeclareMathOperator{\rGr}{Gr-\!}
\DeclareMathOperator{\rmod}{mod-\!}

\DeclareMathOperator{\rMod}{Mod-\!}

\DeclareMathOperator{\lQgr}{\!-Qgr}

\DeclareMathOperator{\rQgr}{Qgr-\!}
\DeclareMathOperator{\rqgr}{qgr-\!}

\DeclareMathOperator{\rTors}{Tors-\!}

\DeclareMathOperator{\rtors}{tors-\!}

\DeclareMathOperator{\GKdim}{GKdim}

\newcommand{\dra}{\dashrightarrow}
\newcommand{\hra}{\hookrightarrow}

\newcommand{\uExt}{\underline{\Ext}}
\newcommand{\uHom}{\underline{\Hom}}
\newcommand{\uEnd}{\underline{\End}}
\DeclareMathAlphabet{\mathpzc}{OT1}{pzc}{m}{it}
 \newcommand{\hilb}{{\sf hilb}\,}
\DMO{\gldim}{gldim}
\DMO{\pdim}{pdim}
\DMO{\injdim}{injdim}
\DMO{\Tot}{Tot}

\newcommand{\wh}{\widehat}
\newcommand{\too}{\longrightarrow}

\title{Ring-theoretic blowing down: I}
\author{D. Rogalski,  S. J. Sierra, and J. T. Stafford}
\address{(Rogalski)
Department of Mathematics, UCSD, La Jolla, CA 92093-0112, USA. }
\email{drogalsk@math.ucsd.edu}
 \address{(Sierra) School of Mathematics,
University of Edinburgh, Edinburgh EH9 3FD, U.K.}
\email{s.sierra@ed.ac.uk}
\address{(Stafford) School of Mathematics,  The University of Manchester,   Manchester M13 9PL,
U.K.}
\email{Toby.Stafford@manchester.ac.uk}

\thanks{The first author is partially supported by NSF grant DMS-1201572. The second author is partially supported by EPSRC grant EP/M008460/1.
  The third author is partially supported by EPSRC grant EP/L018322/1.}

\date{\today}


\begin{document}
 
 \begin{abstract}
  One of the major open  problems in noncommutative algebraic geometry is the classification of noncommutative projective 
  surfaces (or, slightly more generally, of  noetherian connected graded domains of Gelfand-Kirillov dimension 3). Earlier work of the authors
   classified the connected graded noetherian subalgebras of Sklyanin algebras using a noncommutative  analogue of blowing up.
    In order to understand other algebras birational to a Sklyanin algebra, one also needs a notion of blowing down. 
  This is achieved in this paper, where we give a noncommutative analogue of Castelnuovo's classic theorem that $(-1)$-lines on a
   smooth surface can be contracted. 
The resulting noncommutative   blown-down algebra  has pleasant properties; in particular it is always noetherian and is  smooth if the original noncommutative surface is smooth. 
  
  In a companion paper we will use this technique to construct explicit birational transformations between  various noncommutative surfaces which contain an elliptic curve. 

\end{abstract}

\maketitle


\section{Introduction}\label{INTRO}  
Throughout the paper, $\kk$ will denote an algebraically closed field and  all rings will be $\kk$-algebras. 
A $\kk$-algebra $R$ is \emph{connected graded}\label{cg-defn} or \emph{cg} if $R=\bigoplus_{n\geq 0}R_n$ is a finitely 
generated, $\mathbb{N}$-graded algebra with $R_0=\kk$. For such a ring $R$, the category of graded noetherian right 
$R$-modules will be denoted $\rgr R$ with quotient category $\rqgr R$ obtained by quotienting out the Serre subcategory 
of finite dimensional modules. An effective intuition    is  to regard  $\rqgr R$ as the category of coherent sheaves on the (nonexistent) space $\Proj(R)$.

The classification of noetherian, connected graded domains $R$ of Gelfand-Kirillov dimension 3 (or the corresponding noncommutative
   surfaces $\rqgr R$) is one of the major open problems in noncommutative algebraic geometry. This
 has been solved in many particular cases and those solutions have lead to some fundamental advances in the subject; 
 see, for example, \cite{ATV1990,  RSSlong, KRS, SV, VdB-blowups,VdB3} and the references therein. In \cite{Ar}, 
 Artin conjectured that, birationally at least, there is a short list of such surfaces, with the generic case being a Sklyanin 
 algebra. Here, the graded quotient ring $Q_{gr}(R)$ of $R$ is obtained by inverting the non-zero homogeneous elements 
 and  two such  domains  $R, S$ are \emph{birational} if  $Q_{gr}(R)_0 \cong Q_{gr}(S)_0$. Sklyanin algebras are defined in Example~\ref{sklyanin-defn}.

In earlier work of the authors the  connected graded noetherian subalgebras of any  Sklyanin algebra were classified \cite{R-Sklyanin,RSSlong}
and this was achieved through a noncommutative variant of blowing up. However, if one wishes to classify more general algebras 
birational to a Sklyanin algebra one certainly also needs an analogue of blowing down (contracting) exceptional lines. 
This is achieved in this paper.
Before describing these results in  detail we set the stage by reviewing key classical results from commutative algebraic geometry.  
 Thus, let $x$ be a closed point on    a smooth projective surface $X$ over $\kk$, and let $\pi:  Bl_x(X) \to X$ be the blowup of $X$ at $x$. 
These maps, also known as monoidal transformations, are of course fundamental to the birational geometry of surfaces. 

It is well-known that:
\begin{proposition}\label{prop1}
 $Bl_x(X)$ is also a smooth projective surface.  If $L = \pi^{-1}(x)$ is the exceptional locus of $\pi$, then $L \cong \mathbb P^1$ 
 with self-intersection  $L \bigdot L=-1$.
\end{proposition}

A celebrated theorem of Castelnuovo says that the properties of $L$ given in  the proposition also  characterise curves that can be contracted to smooth points.

\begin{theorem}[Castelnuovo]\label{CCC}
Let $Y$ be a smooth projective surface, and let $L$ be a curve on $Y$ with $L\cong \mathbb P^1$ and $L \bigdot L=-1$.
Then there are a smooth projective surface $X$ and a birational morphism $\pi:  Y \to X$ so that $L$ is the exceptional 
locus of $\pi$; in fact $Y \cong Bl_x(X)$, where $x = \pi(L)$.
\end{theorem}

The main aim of this paper is to give noncommutative versions of Proposition~\ref{prop1} and Theorem~\ref{CCC}.
These results apply to a class of  algebras known as   elliptic algebras that occur naturally among algebras birational 
to the Sklyanin  algebras and are defined as follows. An \emph{elliptic algebra} is a connected graded domain $R$ containing a 
central $g \in R_1$ so that $R/(g)$ is isomorphic to a \emph{twisted homogeneous coordinate ring} $B(E,\mathcal{M},\tau)$,
 where $E$ is  an elliptic curve, with an ample invertible sheaf $\mathcal{M}$  
 and infinite order automorphism  $\tau$. 
 We say that $E$ is the {\em elliptic curve  associated} to $R$ and define the  \emph{degree} of $R$  to be the degree of the line bundle $\mathcal{M}$.  
 (See 
Section~\ref{ELLIPTIC ALGEBRAS} for more details.)
 For example, the third Veronese ring  $T = S^{(3)}$ of a Sklyanin algebra $S$  is elliptic; the  Veronese ring  is needed to ensure 
 that the central element  has degree one, but this is a fairly harmless change since   $\rqgr S\cong\rqgr T$.  
Likewise, if $S'$ is a 3-dimensional cubic Sklyanin algebra, as discussed in  Example~\ref{sklyanin-defn},
  then $T'= (S')^{(4)}$ is elliptic.  The space $\rqgr S$ should be thought of as a noncommutative
 $\PP^2$, while $\rqgr S'\cong\rqgr T'$ should be thought of as a noncommutative version of $\PP^1 \times \PP^1$.

An appropriate  noncommutative analogue of a monoidal transformation of an elliptic algebra $R$  is   known.  
In more detail,  $R_1/g\kk\cong B(E,\mathcal{M},\tau)_1$ may be identified with global sections of the invertible sheaf $\sM$.  
If $p\in E$  and $ \deg \sM\geq 3$,   the \emph{blowup of $R$ at $p$}  is defined to be   the subalgebra $P=Bl_p(R)$ of $R$ 
generated by those elements of $R_1$ whose images mod $g$ vanish at $p$. By  \cite[Theorem~1.1]{R-Sklyanin},
$P$ is again an elliptic algebra and, moreover, has   properties analogous to those of a (commutative) blowup. 
In particular, it has an analogue of an exceptional line. To be precise,  a graded  $P$-module $L=\bigoplus_{n\in \mathbb{Z}}L_n$
 is a \emph{line module} if $L$ is cyclic with  Hilbert series $\hilb L=\bigoplus_{n\in \mathbb{Z}}(\dim_{\kk}L_n)s^n=1/(1-s)^2$.  
 Then $P=Bl_p(R)$ does indeed have a distinguished line  module $L$, called the {\em exceptional line module} and  
 characterised by the fact that $R/P \cong \bigoplus_{i\geq 1} L[-i]$ as $P$-modules.   Inducting on this
 procedure one can blow up as many as seven points on the noncommutative projective plane $\rqgr T$ (one can even blow up 
 eight points although the definition is more subtle since the ring is no longer generated in degree one \cite{RSSshort}).

We would like to reverse this procedure.  
For a noncommutative version of Castelnuovo's Theorem~\ref{CCC}, we need not only an analogue of a line but also an analogue of self-intersection.
If $\rqgr P$ is a  \emph{smooth} noncommutative scheme in the sense that  the category  $\rqgr P$  has finite homological dimension, 
then  an appropriate notion of  intersection number is
 $$(M\dotms N  )=\sum_{n\geq 0} (-1)^{n+1} \dim_\kk \Ext^n_{\rqgr P}(M,N),$$
    for line modules $M$ and $N$ (see \cite{MS}).  Unfortunately even if   $\rqgr R$ is smooth, if $P = Bl_p(R)$ then the blowup 
    $\rqgr P$ need not be smooth, in which case the 
 self-intersection $(L\dotms L  )$ of a line module $L$ can be undefined. (See Section~\ref{EXAMPLE}, where an example is 
 constructed by blowing up an elliptic algebra twice at the same point.)
So, we use a weaker notion of intersection number, defined as follows.  Let $\grk$ denote the torsion-free rank of a finitely 
generated $\kk[g]$-module.  Then, 
it is not hard to show that $(M\dotms N  )
=\sum_{n\geq 0} (-1)^{n+1}\grk \Ext^n_{P}(M,N)$  (combine Proposition~\ref{prop:0.6} and  Lemma~\ref{lem0.1}, in the 
notation from  the beginning of Section~\ref{BACKGROUND}).  
Moreover,  as is discussed in Sections~\ref{INTERSECTION THEORY} and~\ref{INTLINE},
the simpler sum 
$$(M\bigdot N) =-\grk \Hom_{P} (M,N) + \grk \Ext^1_{P} (M,N)$$  is a satisfactory alternative to $(M\dotms N  )$.

It will actually  be convenient to use the following, still weaker concept.  
Assume that $L$ is a line module over an elliptic algebra 
 $P$ and write $L=P/J$ for the \emph{line ideal $J$}.  We say two graded, locally finite dimensional  
vector spaces  $M$ and $N$ are \emph{numerically equivalent} if they have the same Hilbert series:  
$\hilb M = \hilb N$.  
Then the relevant condition is:
 \begin{equation}\label{Hilb-intersection}\tag{*} \text{For a line module $L = P/J$, the rings $P$ and $\End_{P}(J) $ are numerically equivalent.} 
 \end{equation}
This notion is appropriate, as the next result shows.

\begin{proposition}[Theorems~\ref{thm:one} and \ref{thm:inverse}]
\label{intro-1}
{\rm (1)}  Let $P$ be an elliptic algebra such that $\rqgr P$ has finite homological dimension, and let $L$ be a line module over $P$. Then   $(L\dotms L)=-1$ $\iff$
 $(L\bigdot L)=-1$ $\iff$  $(\ast)$ holds.

{\rm (2)} If $P=Bl_p(R)$ is the blowup of an elliptic algebra $R$, then $(\ast)$ holds for the exceptional line module~$L$.
\end{proposition}

Our definition of self-intersection leads to a noncommutative version of Castelnuovo's Theorem~\ref{CCC}, as we next show.
\begin{theorem}[Theorems~\ref{thm:I.1} and~\ref{thm:inverse}]\label{intro-2}

\begin{enumerate} 

\item  Let $P$ be an elliptic algebra with associated elliptic curve $E$ and let $L$ be a line module with $(L\bigdot L)=-1$ 
or, more generally, one that satisfies $(\ast)$.  Then one can blow down the line $L$.

More precisely, there are an elliptic algebra $R=\wt{P} \supseteq P$, again associated to $E$, and a point $p \in E$ so
 that $P \cong Bl_p(R)$, with exceptional line $L$.  

\item Conversely, if $Q$ is an elliptic algebra of degree $\geq 4$, then blowing $Q$ up at a point $p$ of the associated elliptic 
curve $E$ and blowing down the exceptional line of $Bl_p(Q)$ returns the algebra~$Q$.  
\end{enumerate}
\end{theorem}

\noindent
{\bf Definition.} The ring $\wt{P}$ from part (1) of the theorem is called the \emph{blowdown} of $P$ at $L$.  
 \medskip
 
The key step in the  proof of part~(1) of the theorem is 
to show that there exists a  right $P$-module $M$ with  $Q_{gr}(P)\supset M\supset P$ for which    
$M/P\cong \bigoplus_{i\geq 1} L[-i]$. One then shows that $M$ is actually a ring with the properties specified by the theorem. 

Elliptic algebras have a number of pleasant properties; for example they are automatically noetherian and 
satisfy the Artin-Schelter Gorenstein and Cohen-Macaulay conditions (see Proposition~\ref{prop:elliptic}). Thus, in particular, these
 conditions hold for the blowdown of an elliptic algebra. More subtly we have an analogue of the smoothness part of Castelnuovo's Theorem~\ref{CCC}.

\begin{theorem}[Corollary~\ref{cor:smooth}]\label{intro-3}
Let $P$ be an elliptic algebra and suppose that $L$ is a line module  
satisfying $(\ast)$, with blowdown $\wt{P}$.
 Assume, moreover, that   $L[g^{-1}]_0$ has finite projective dimension over $P[g^{-1}]_0$.  
 
 Then the  noncommutative scheme
    $  \rqgr \wt{P}$ is smooth if and only if  $\rqgr P$ is smooth.
\end{theorem}

Our eventual goal is to classify graded algebras birational to a Sklyanin algebra.
Using the commutative geometry of surfaces as a guide, one would presumably need to classify ``minimal models'' (in the 
appropriate sense) and to show that any reasonable algebra in this class can be blown down to a minimal model.  Clearly, 
the noncommutative versions of $\PP^2$ and  Van den Bergh's quadrics \cite{VdB3}  should be minimal, and in forthcoming 
work we show that this is true   \cite{RSSminimal}.
We do not yet know whether  these are the only minimal models.

We also do not know how to show that any algebra birational to a Sklyanin algebra can be blown down to give a minimal 
model.  In the birational theory of commutative surfaces, this is proved  using the following consequence of Zariski's Main Theorem:

\begin{theorem}[Zariski]  \label{thm:Z}
 Let $X \dra Y$ be a birational map of smooth projective surfaces.
Then there are a smooth projective surface $Z$ and compositions of monoidal transformations $Z \to X$, $Z \to Y$ so that 
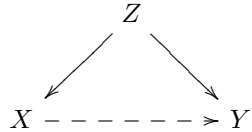
\[\xymatrix{
 & Z \ar[dr] \ar[dl] & \\ X \ar@{-->}[rr] && Y}
\]
commutes. 
\end{theorem}

As yet, there is no noncommutative  analogue of Theorem~\ref{thm:Z} although in  the companion paper \cite{RSS4}, we do prove:

\begin{theorem}
 Let $E$ be the elliptic curve associated to the cubic Sklyanin algebra $S'$, as defined above, and let $r \in E$ be generic.  
 Then there is a Sklyanin algebra $S$   associated to $E$ and points $p, q \in E$ so that
\[ Bl_r((S')^{(4)}) \cong Bl_{p,q}(S^{(3)}).\]
\end{theorem}
 In fact this theorem also holds when $S'$ is replaced by any generic noncommutative  quadric surface in the sense of \cite{VdB3}.
 This theorem is a noncommutative version of the isomorphism $Bl_{p,q}(\PP^2) \cong Bl_r(\PP^1\times \PP^1)$ arising from Theorem~\ref{thm:Z}. 
 The birationality of $S$ and $S'$   was first proved by Van den Bergh, with a detailed proof given in \cite{PV}. 
 
The paper is organised as follows.  
In Section~\ref{BACKGROUND} we review background on twisted homogeneous coordinate rings of elliptic curves.  
In Section~\ref{POINTS} we study point modules over such a ring.
In Section~\ref{ELLIPTIC ALGEBRAS} we define elliptic algebras and give their basic properties, and in
 section~\ref{LINES} we define and study line modules over elliptic algebras.  
In Sections~\ref{INTERSECTION THEORY} and~\ref{INTLINE} we develop noncommutative intersection theory and prove part (1) of Proposition~\ref{intro-1}.
In Section~\ref{BLOWDOWN} we prove our main blowing down Theorem~\ref{intro-2}, and in Section~\ref{SMOOTH} we prove Theorem~\ref{intro-3}.
Finally, in Section~\ref{EXAMPLE}  we study the effect of blowing up the same point twice.

\bigskip

\noindent {\bf Acknowledgements: }

We would like to thank Ken Goodearl, Dennis Presotto, and Michel Van den Bergh for useful discussions and comments. 
We also thank the referee of the first version of this article for their careful reading and helpful comments.

Part of this material is based upon work supported by the National Science Foundation under Grant No. 0932078 000, while
 the authors were in residence at the Mathematical Sciences Research Institute (MSRI) in Berkeley, California, during the spring semester of 2013. 
During this visit,  Rogalski was partially supported by NSF grant DMS-1201572, Sierra was supported by the Edinburgh Research
 Partnership in Engineering and Mathematics, while Stafford was partially supported by the Clay Mathematics Institute and Simons 
 Foundation.  We thank  all these institutions for their support.


\section{Basic concepts}\label{BACKGROUND}

In this section we review some basic material, including twisted homogeneous coordinate rings, that will be used frequently and without particular comment throughout the paper.

Throughout we work over an algebraically closed field $\kk$, and rings will   be $\kk$-algebras unless otherwise noted.
Given a noetherian $\mb{N}$-graded $\kk$-algebra $A$, let $\rGr A$ be the category of $\mb{Z}$-graded right $A$-modules, 
with morphisms $ \Hom_{\rGr A}$ preserving degree.
Write  $\rgr A$ for the full subcategory of noetherian modules.  Let $[1]: \rGr A \to \rGr A$ be the \emph{shift functor}: \label{shift-defn}
the autoequivalence sending $M=\bigoplus M_i \to M[1]=\bigoplus M[1]_i $, where $M[1]_n = M_{n+1}$.  For $M, N \in \rGr A$  
  the \emph{graded Hom} groups are
$\uHom_A(M,N) = \bigoplus_{n \in \mb{Z}} \Hom_{\rGr A}(M,N[n])$, with 
derived functors  $\uExt_A(M,N) = \bigoplus_{n \in \mb{Z}} \Ext^i_{\rGr A}(M,N[n])$.
 If $A$ and $M$ are  noetherian, then $\uHom_A(M,N) $ equals 
 the usual ungraded Hom which will always be written $\Hom_{A}(M,N)$. 
Similarly, 
$\uExt^i_A(M,N) = \Ext^i_{A}(M,N)$, see \cite[Proposition 3.1]{AZ}. Finally, set 
  $\uEnd_A(M)=\uHom_A(M,M)$.  We will frequently use the fact that 
\begin{equation}\label{elementary-shift}
\Ext^r_{\rGr A}(M[-n],N)=\Ext^r_{\rGr A}(M,N[n])=\Ext^r_{\rGr A}(M,N)[n],\qquad\text{for any $M,N,n$ and $r$.}
\end{equation}

Let $A=\bigoplus_{n\geq 0}A_n$ be a cg noetherian algebra, and note that $A$ is necessarily \emph{locally finite}
 in the sense that  $\dim_{\kk}A_n<\infty$ for all $n$.
Let $\rtors A$ be the category of modules in $\rgr A$ which 
are finite-dimensional over $\kk$, and let $\rTors A$ be the subcategory of $\rGr A$ consisting of direct limits of finite-dimensional 
modules.  Write  $\rQgr A$ for the quotient category $\rQgr A/\rTors A$, with quotient functor $\pi: \rGr A \to \rQgr A$.
Then $\rqgr A = \rgr A/\rtors A$ is identified with the noetherian objects in $\rQgr A$.
Following \cite{AZ}, the pair $(\rQgr A, \pi(A))$ is called the \emph{noncommutative projective scheme} associated to $A$.
 The autoequivalence $[1]$ of $\rGr A$ induces an autoequivalence, again written $[1]$,  of $\rQgr A$.
 We again have  graded Hom groups $\uHom_{\rQgr A}(\mc{M}, \mc{N}) = 
\bigoplus_{n \in \mb{Z}} \Hom_{\rQgr A}(\mc{M}, \mc{N}[n])$.

We emphasise here that  a graded module is called  \emph{torsionfree} provided it has no finite dimensional submodules.
In contrast a module $M$ (graded or not) over a prime ring $R$ is called \emph{Goldie torsionfree}
\label{tors-defn} if no element of the module is killed by a regular element of the ring.
 
Next we review some important homological conditions.  

\begin{definition}\label{def:gor}
A ring $A$ is called  \emph{Auslander-Gorenstein} if 
\begin{enumerate}
\item[(i)]   $\injdim(A)<\infty$, in the sense that $A$ has finite injective dimension on both left and right; 
\item[(ii)] if $0 \leq p<q$ and  $M$ is a finitely generated $A$-module, then
$\Ext_{A}^p(N,\,A)=0$ for every submodule $N$ of 
$\Ext_{A}^q(M,\,A)$.
\end{enumerate}

Write $\GKdim(M)$ for the Gelfand-Kirillov dimension of an $A$-module $M$, as in \cite{KL}. 
An $R$-module $M$  is called  \emph{$d$-pure} if \label{pure-defn}
$\GKdim N = d=\GKdim M$ for all nonzero submodules $N$ of $M$, and   is \emph{$d$-critical} if 
$\GKdim M/N < d$ for all all nonzero submodules $N$ of $M$. 
Let $A$ be a noetherian Auslander-Gorenstein $\kk$-algebra with $\GKdim(A)<\infty$. 
For an $A$-module $M$, write
 $j(M)= \min\{ r : \Ext^r_{A}(M,A)\not= 0\}$ for the \emph{homological grade} of $M$.
The algebra $A$ is called \emph{Cohen-Macaulay} (or CM), provided 
 that  $j(M)+\GKdim(M)=\GKdim(A)$ holds for every finitely generated
 $A$-module $M$.    The module $M$ is then called 
 \emph{Cohen-Macaulay} (or CM) if 
$\Ext_{A}^r(M,A) = 0$ for all $r \neq j(M)$.
  
Finally, a cg noetherian $\kk$-algebra $A$ is called \emph{Artin-Schelter (AS) Gorenstein} if 
$d=\injdim(A)<\infty$ and 
$\uExt_A^j(\kk, A)  \cong  \delta_{j,d}\kk[\ell] $,
where $\kk = A/A_{\geq 1}$ is the \emph{trivial module}, and $\ell$ is some shift of grading.
\end{definition}

 Let $X, Y$ be $\kk$-schemes. Write  $\Qcoh X$ for  the category of quasi-coherent sheaves 
on $X$, with $\coh X$ the subcategory of coherent sheaves.
Given a morphism of $\kk$-schemes $\phi: X \to Y$ and   $\mc{F}\in \Qcoh Y$, we write  $\mc{F}^{\phi}$ for the pullback $\phi^*(\mc{F})$.
If $\mc{L}$ is an invertible sheaf on $X$, and $\tau \in \Aut_{\kk}(X)$ is a $\kk$-automorphism, one defines
 the \emph{TCR or twisted homogeneous coordinate ring}\label{TCR-defn}
 $B(X, \mc{L}, \tau)= 
\bigoplus_{n \geq 0} H^0(X, \mc{L}_n)$, where $\mc{L}_n = \mc{L} \otimes \mc{L}^{\tau} \otimes \dots \otimes \mc{L}^{\tau^{n-1}}$.
This is a $\mb{N}$-graded $\kk$-algebra with multiplication defined as follows:  for $x \in B_m, y \in B_n$, then
$x \star y = \mu(x \otimes y^{\tau^m})$, where $\mu:  H^0(E, \mc{L}_m) \otimes H^0(E, \mc{L}_n^{\tau^m}) \to B_{n+m} = H^0(E, \mc{L}_{n+m})$ is the obvious   map.

In this paper we are primarily concerned with the TCR
 $B(E, \mc{L}, \tau)$ of an    elliptic curve $E$.   
In this case, the following result is well-known:

\begin{lemma}\label{lem:basics for B}
 Let $E$ be an elliptic curve over $\kk$ and $\sL, \sM$ be invertible sheaves on $E$ of degree $\geq 2$.  
\begin{enumerate}\item 
The natural map 
\[ \mu: H^0(E, \sL)\otimes H^0(E, \sM) \to H^0(E, \sL \otimes \sM)\]
is surjective unless $\sL \cong \sM$ has degree 2, in which case $\dim \coker \mu = 1$.
\item   Let  $\tau\in\Aut_{\kk}(X)$ have infinite order.  
Then $B = B(E, \mc{L}, \tau)$ is generated as an algebra in degree~$1$. 
\item If $B$ is as in (2), then $B$ is a noetherian domain which is Auslander-Gorenstein, CM, and AS-Gorenstein, with 
$\operatorname{injdim} B_B = 2$ and $\uExt^j_B(\kk, B) = \delta_{2j} \kk$.  
\item If $B$ is as in (2), the map  $\mc{F} \mapsto \pi( \bigoplus_{n \geq 0} H^0(E, \mc{F} \otimes \mc{L}_n))$ 
defines   an equivalence of categories 
$\Qcoh E \to \rQgr B$.

\end{enumerate}
\end{lemma}
\begin{proof}
Parts (1) and (2) are \cite[Lemma 3.1]{R-Sklyanin}.  That $B$ is a domain follows immediately from the definition, and the noetherian property 
is \cite[Theorem~1.4]{AV}, while part (4) is  \cite[Theorem~1.3]{AV}.
  The remaining homological properties follow from \cite[Theorems~6.3 and 6.6]{Lev1992}.    (Levasseur assumes that $\deg \mc{L} \geq 3$, but the proof only uses that $\sL$ is ample).
\end{proof}

\begin{notation}\label{section-functor} 
The quotient functor $\pi: \rGr B \to \rQgr B$ has a right adjoint $\omega: \rQgr B \to \rGr B$ called 
the section functor, which may be described more explicitly as follows.   
If $M = \bigoplus_{n \geq 0} H^0(E, \mc{F} \otimes \mc{L}_n)$ for a coherent sheaf $\mc{F}$, then 
$\omega \pi(M) = \bigoplus_{n \in \mb{Z}} H^0(E, \mc{F} \otimes \mc{L}_n)$, where we define 
$\mc{L}_n$ for $n < 0$ by $\mc{L}_n = (\mc{L}^{\tau^{n}} \otimes \dots \otimes \mc{L}^{\tau^{-1}})^{-1}$.
We say that a graded $B$-module $M$ is \emph{saturated} if it is in the image of the section functor $\omega$.  By\cite[(2.2.3)]{AZ}, this is 
equivalent to $\uExt^1_B(\kk, M) = 0$.  
\end{notation}
 
Given an $\mb{N}$-graded noetherian domain $A$, the localisation  of $A$  at the set of  nonzero homogeneous elements exists 
and is called the \emph{graded quotient ring} $Q = Q_{\gr}(A)$ of $A$. \label{quot-defn}
 Given noetherian graded $A$-submodules $M,N$ 
of $Q$, we identify $\uHom_A(M,N) $ with $  \{x \in Q : xM \subseteq N \}$.  In particular,   
 $\uHom_A(M,A)$ is identified with $M^*  = \{ x \in Q : xM \subseteq A \}$ and  $M$ is   \emph{reflexive} if $M=M^{**}$.
If $A=B = B(E, \mc{L}, \tau)$ as in Lemma~\ref{lem:basics for B}, then $Q_{gr}(B) \cong \kk(E)[t, t^{-1}; \tau]$ where 
$\kk(E)$ is the function field of $E$ with the induced action of  $\tau$.  We sometimes fix an isomorphism 
$Q_{\gr}(B) \cong Q = \kk(E)[t, t^{-1}; \tau]$, and write $B$ as the explicit subalgebra 
$B = \bigoplus_{n \geq 0} H^0(E, \mc{L}_n) t^n$ of $Q$, where each $H^0(E, \mc{L}_n)$ is then given a fixed embedding
 into $\kk(E)$.  The following result will be useful in calculating Homs between $B$-submodules of $Q$.

\begin{lemma}
\label{lem:uhomB}
Let $B = B(E, \mc{L}, \tau) = \bigoplus_{n \geq 0} H^0(E, \mc{L}_n) t^n \subseteq Q = \kk(E)[t, t^{-1}; \tau]$ for an elliptic 
curve $E$ over $\kk$, with invertible sheaf $\mc{L}$ of degree $\geq 2$ and infinite order automorphism $
\tau\in  \Aut_{\kk}(E)$.   Let $\mc{F}$ and $\mc{G}$ be invertible $\mc{O}_E$-subsheaves of the constant sheaf   $\kk(E)$,  
and let $M = \bigoplus_{n \in \mb{Z}} H^0(E, \mc{F} \otimes \mc{L}_n)t^n$, $N = \bigoplus_{n \in \mb{Z}} H^0(E, \mc{G} \otimes \mc{L}_n)t^n$
 be saturated $B$-submodules of $Q$.  Then 
\[
\uHom_B(M,N) = \bigoplus_{n \in \mb{Z}} H^0(E, (\mc{F}^{\tau^{n}})^{-1} \otimes \mc{G} \otimes \mc{L}_n)t^n \subseteq Q.
\]
\end{lemma} 

\begin{proof}   This is similar to the proof of   \cite[Lemma 6.14(i)]{RSSlong}, but since we use the result frequently we give the details.

 Write $ H= \uHom_B(M,N)$, and $X= \bigoplus_{n \in \mb{Z}} H^0(E, (\mc{F}^{\tau^{n}})^{-1} \otimes \mc{G} \otimes \mc{L}_n)t^n $
 both of which   can be identified with   subspaces of $Q$.
  For each $n$,  
let $\mc{H}_n$ be the subsheaf of the constant sheaf  $\kk(E)$ generated by   $H_n t^{-n} \subseteq \kk(E)$. 
Let $\mc{M}_n  =  \mc{F}  \otimes  \mathcal{L}_n$; thus $M_n = H^0(E, \mc{M}_n)$, and $M_n$ generates 
the sheaf $\mc{M}_n$, for $n\gg0$, say for $n \geq n_0$, because $\mc{L}$ is $\tau$-ample by \cite[Proposition 1.5]{AV}.  
Similarly, write $\mc{N}_n  =  \mc{G}  \otimes  \mathcal{L}_n$  for all $n$.

For $n \geq n_0$ and $r \geq 0$, the equation $H_r M_n \subseteq N_{n+r}$ forces 
$\mc{H}_r \mc{M}_n^{\tau^r} \subseteq \mc{N}_{n+r}$ and so
\[ 
  \mc{H}_r \otimes \Bigl( \mc{F} \otimes \mathcal L_n \Bigr)^{\tau^r}  
\   \subseteq  \        \mc{G} \otimes \mathcal L_{n+r}.
\]
Equivalently
$\mc{H}_r \subseteq  (\mc{F}^{\tau^{r}})^{-1} \otimes \mc{G} \otimes \mc{L}_r$  and so  
 $H\subseteq \bigoplus H^0(E, \mc{H}_r)t^r \subseteq X$.
Another calculation shows that 
$\big( (\mc{F}^{\tau^{r}})^{-1} \otimes \mc{G}\otimes \mc{L}_r  \bigr) \, \mc{M}_n^{\tau^r} = \mc{N}_{n+r}$ 
for $r, n \geq 0$ and taking sections for $n \geq n_0$ shows that 
$ X \subseteq \uHom_B(M_{\geq n_0}, N)$.

To complete the proof we need to prove that $H =  \uHom_B(M_{\geq n_0}, N)$.    Clearly, $H \subseteq \uHom_B(M_{\geq n_0}, N)$.  
However, if 
$\theta\in \uHom_B(M_{\geq n_0}, N)_t$ for some $t$, then we may consider $\theta$ as an element of $Q$. 
We see that  $Z=(\theta M+N)$ is  a $B$-submodule of $Q$ such that $Z B_{\geq n_0}\subseteq N$.
Since $N$ is saturated, this forces $Z\subseteq N$ and  $\theta \in H$, as desired.
\end{proof}
 
\begin{notation}\label{loc-fin}
Let $A$ be connected graded.   A $\mb{Z}$-graded $A$-module $M$ is 
 \emph{left} respectively \emph{right bounded} if $M_n = 0$ for $n \ll 0$, respectively  for $n \gg 0$.  Obviously right bounded modules are in $\rTors A$, and finitely generated 
graded modules $M$ are left bounded.  Importantly, if $A$ is noetherian and $M$ and $N$ are finitely generated graded $A$-modules, 
then by considering a free resolution of $M$, each $\uExt^i_A(M,N)$ is   left bounded and locally finite.
 If $M$ is locally finite, the \emph{Hilbert series}\label{hilb-defn} of $M$ 
is the formal Laurent series $\hilb M=h_M(s)=\sum_{n \in \mb{Z}} \dim_{\kk} M_n s^n \in \ZZ((s))$.  
Note that we use the  notations $\hilb M$ and $h_M$ interchangeably.
Given two Hilbert series $g(s) = \sum a_n s^n $ and $ h(s) = \sum b_n s^n$ we write $g(s) \leq h(s)$ 
 if $a_n \leq b_n$ for all $n \in \mb{Z}$.
 \end{notation}
 

\section{Point modules}\label{POINTS}

In this section we study some homological properties of point modules over a  twisted homogeneous coordinate ring. 
Throughout the section, we fix an elliptic curve $E$,  an automorphism  $\tau\in  \Aut_{\kk}(E)$ of infinite order and 
an invertible sheaf $\mc{M}$ with $\deg \mc{M}\geq 3$, although many of the results hold more generally. Corresponding
 to this data, set $B = B(E, \mc{M}, \tau)$. Points of $E$ will always mean closed points.

\begin{definition}\label{point-defn}
Let $A$ be a cg $\kk$-algebra that is generated in degree one. Then 
  a \emph{point module} over $A$   is 
a graded cyclic module $M$ with Hilbert series $h_M(s) = 1/(1-s)$.  If $M$ is an  $A$-point module, then 
there is a graded isomorphism $M \cong A/I$ for a unique right ideal $I$, called
a (right) \emph{point ideal}.
Now let $B = B(E, \mc{M}, \tau)$ be a TCR as defined in the last section;  thus $B$ is generated in degree $1$ by Lemma~\ref{lem:basics for B}. 
By  Lemma~\ref{lem:basics for B}(4),  the isomorphism classes of $B$-point modules  are in one-to-one correspondence 
with the closed points of $E$; explicitly, if $p \in E$ with skyscraper sheaf $\mc{O}_p$ and  ideal sheaf  $\mc{I}_p$, 
then $p \in E$ corresponds to the point module $M_p = \bigoplus_{n \geq 0} H^0(E, \mc{O}_p \otimes \mc{M}_n)$, with   point ideal
$I_p = \bigoplus_{n \geq 0} H^0(E, \mc{I}_p \otimes \mc{M}_n) \subseteq B$.  It is easy to see that $I_p$ is a saturated right ideal in the sense of Notation~\ref{section-functor}.

When considering shifts of point modules, the following formula will be useful.  
\begin{equation}
\label{eq:truncshift}   M_p[n]_{\geq 0} \cong M_{\tau^n p}
 \qquad \text{for any } p\in E\ \text{and } n \geq 0
\end{equation}
(see, for example, \cite[Lemma~4.8(1)]{RSSshort}).
In particular, $(M_p)_{\geq n} \cong (M_{\tau^np})[-n] $ for all  $p\in E$ and $n\geq 0$. 

\end{definition}

\begin{remark}\label{leftpoint-remark}
We   occasionally  work with left point modules over $B = B(E, \mc{M}, \tau)$.  Of course there are left-sided versions of 
all of the results above; in particular, the left $B$-point modules are again  in bijection with the points of $E$.  In this case the equivalence of 
categories $\Qcoh E \to B \lQgr$ is induced by the functor 
$\mc{F} \mapsto \bigoplus_{n \geq 0} H^0(E, \mc{F}^{\tau^{n-1}} \otimes \mc{M}_n)$.  In particular, 
the    left point module corresponding to $p\in E$ is 
$M^{\ell}_p = \bigoplus_{n \geq 0}H^0(E, (\mc{O}_p)^{\tau^{n-1}} \otimes \mc{M}_n)$. 
Moreover,  $M^{\ell}_p \cong B/J_p$ for the 
\emph{left point ideal} 
$J_p = \bigoplus_{n \geq 0}  H^0(E, (\mc{I}_p)^{\tau^{n-1}} \otimes \mc{M}_n)$.  Note that the correspondence is set up so that if $I_p$ is the right 
point ideal corresponding to $p$, and $J_p$ the left point ideal, then $(I_p)_1 = (J_p)_1 = H^0(E, \mc{I}_p \otimes \mc{M})$.
We also have  the following analogue of \eqref{eq:truncshift}:
\beq\label{leftpoint}
(M^{\ell}_q)[n]_{\geq 0} \cong M^{\ell}_{\tau^{-n}q} 
\qquad \text{for any } q\in E\ \text{and } n \geq 0.
\eeq
\end{remark}

\begin{lemma}\label{lem:Bdual}
 Let $B = B(E, \sM, \tau)$ as above with  a  right  point module  $M=M_p$.   
Then $M$ is CM and 
 $\uExt^1_B(M_p, B)[1]\cong M^\ell_{\tau^{-2}(p)} $ is a  left  point module.  
 The analogous result holds for left point modules.
 \end{lemma}

\begin{proof} We freely use the properties of $B$ given by   Lemma~\ref{lem:basics for B}.
Set $E^{pq}(M)=\uExt_B^p(\uExt_B^q(M,B),B)$. 

 We first show that $M$ is CM, which will be a routine consequence of the spectral sequence 
\begin{equation}\label{spectral-eq}
E_2^{p,-q}=E^{pq}(M) \Rightarrow \mathbb{H}^{p-q}(M)=\begin{cases} M & \text{ if } p=q\\
0& \text{ otherwise,} \end{cases}
\end{equation}
as described   in \cite[Theorem~2.2(a)]{Lev1992} or \cite[\S I.1]{Bj}. 
Note that, as $B$ is generated in degree one, every proper factor module of $M$ is finite dimensional, and so $M$ is 1-critical.  
Now $B$ is Auslander-Gorenstein of  injective dimension 2.
Therefore, by   the 1-criticality of $M$ and \cite[Theorem~2.4(b)]{Lev1992}, and  in the notation of   that result,
  $F^j(M)=0$ for $j>1$. Hence   $M=E^{11}(M)$ but $E^{22}(M)=0$ 
  by \cite[Theorem~2.2(b)]{Lev1992}. In particular, $\uExt^2_B(M,B)=0$  by the Gorenstein property while $\uExt^0_B(M,B)=0$ since $B$ is a domain.  Thus $M$ is indeed CM.

It is almost immediate from \eqref{spectral-eq} that $\uExt^1_B(M, B)[1]$ is a point module, but as we need to identify the corresponding point, we take a different approach.  Write $M=M_p=B/I_p$. Applying $\uHom_B( -, B)$ to the exact sequence $0 \to I_p \to B \to M_p \to 0$  shows that 
$\uExt^1_B(M_p, B) \cong I_p^*/B$.  
By Lemma~\ref{lem:uhomB}   
\[
I_p^* \cong \bigoplus_{n\geq 0}H^0(E, (\mc{I}_p^{\tau^n})^{-1} \otimes \sM_n) = 
\bigoplus_{n\geq 0}H^0(E, \sM_n(\tau^{-n}(p))),
\]
and so  $I_p^*/B$ has Hilbert series $s/(1-s)$.  It follows from Lemma~\ref{lem:basics for B} 
that $B (I_p^*)_1 = (I_p^*)_{\geq 1}$, and  so $I_p^*/B$ is cyclic. In other words,   
 $\uExt^1_B(M_p, B)[1]$ is a point module. The fact that this point module is indeed
 $  M^\ell_{\tau^{-2}(p)}$ now follows from Remark~\ref{leftpoint-remark} and \eqref{leftpoint}. 
The result for left modules is left to the reader. \end{proof}
 
We next  want to compute the $\uExt$ groups between point modules over $B = B(E, \mc{M}, \tau)$. As the next result shows, in $\rqgr B$ this  
follows easily from the equivalence $\rqgr B \simeq \coh E$.  
When there is no chance of confusion, given  $M \in \rgr B$, the object $\pi(M)\in \rqgr B$  will also be written 
 as $M$.

\begin{lemma}
\label{lem:extqgr} Let $B = B(E, \mc{M}, \tau)$ as before.
Then
$ \uExt^m_{\rqgr B}(M_p,\,M_{q}) = 0$ for $m\geq 2$ and 
\[
\uHom_{\rqgr B}(M_p, M_q) \cong \uExt^1_{\rqgr B}(M_p,M_q) \cong \begin{cases}  0 & \text{ for } p, q\ \text{on distinct orbits} \\
\kk[-j] & \text{ if } p = \tau^j(q),  \text{ for }  j \in \mb{Z}. \end{cases}
\]
\end{lemma}
\noindent
{\bf Remark:} This   result and its proof also hold when  $\deg\mc{M}=2$.

\begin{proof} 
As noted in Definition~\ref{point-defn}, the point module
 $M_p$ corresponds to the skyscraper sheaf $\mathcal{O}_p$ at $p$.  Thus, by \eqref{eq:truncshift},
\[
\uExt^m_{\rqgr B}(M_p,\,M_q) = \bigoplus_{n \in \ZZ} \Ext^m_{\rqgr B}(M_p,\,M_q[n]) =  
\bigoplus_{n \in \ZZ} \Ext_{E}^m(\mathcal{O}_p, \, \mathcal{O}_{\tau^{n}(q)})
\]
 for all $m$.   
 Since $E$ is a smooth curve, $\Ext_{E}^m(\mathcal{O}_p,\, \mathcal{O}_r)= 0$ for 
   any closed point $r\in E $ and  
$m \geq 2$.
On the other hand, working locally  gives 
\[
\Hom_{E}(\mathcal{O}_p,\, \mathcal{O}_r) = \Ext_{E}^1(\mathcal{O}_p,\, \mathcal{O}_r) = \begin{cases} 0  & 
 \text{ if }  p \neq r \\ \kk &  \text{ if }  p=r. \end{cases}
\]
Now apply \eqref{elementary-shift}.
\end{proof}

We next want to   prove the analogue of Lemma~\ref{lem:extqgr} for homomorphisms in $\rgr B$, for which we need several  elementary observations.

\begin{lemma}\label{sub-ext}  Let $q\in E$. Then the following hold.
\begin{enumerate}
 \item The only   torsionfree extensions of $M_q$ by finite dimensional graded $B$-modules
are the shifted  point modules $M_{\tau^{-n}(q)}[n]$ for $n\geq 0$.  
\item $\uExt^1_B(\kk[-n],M_q)=\kk[n+1]$ for all $n \in \ZZ$.
\item For all $p,q\in E$, one has $\uExt^1_B(M_p, M_q)_{-1} \neq 0$.
\item $\Ext^1_{\rGr B}(M_p,M_p)\not=0.$
\end{enumerate}
\end{lemma}

\begin{proof}
(1) Note first that, by \eqref{eq:truncshift}, the module $X=M_{\tau^{-n}(q)}[n]$ satisfies 
$X_{\geq 0} = \big((M_{\tau^{-n}(q)})_{\geq n}\big)[n]=M_q,$ and so $X$ is an extension of the required form.  
Conversely, any such extension is necessarily a 1-critical module, and so uniqueness follows, for example, from \cite[Corollary~3.7(1)]{RSSlong}. 

(2) Since any non-trivial extension of $M_q$ by a shift of $\kk$ is necessarily torsionfree, it follows from (1) that the only   such extension is 
$ 0 \to M_q \to M_{\tau^{-1} q}[1] \to \kk[1] \to 0. $ Thus $\uExt^1_B(\kk[-n], M_q)=\kk[n+1].$

(3) Consider 
$B/J$, where $J$ is the right ideal $\bigoplus_{n \geq 0} H^0(E, \mc{M}_n(-p-\tau^{-1}(q)))$. 
Then the point ideal $I_p = \bigoplus_{n \geq 0} H^0(E, \mc{M}_n(-p))$ contains $ J$, with 
 $I_p/J \cong (M_{\tau^{-1}(q)})_{\geq 1} \cong M_q[-1]$.  
It is clear that this extension of $I_p/J$ by $B/I_p\cong M_p$  is nonsplit, since $B/J$ is cyclic.

(4)  By Lemma~\ref{lem:basics for B}(4), 
the nonsplit extension of $\mc{O}_p$ by itself in $\coh E$  
gives a nonsplit extension $0 \to \pi(M_p) \to \mc{F} \to \pi(M_p) \to 0$ in $\rqgr B$, for 
some $\mc{F} \in \rqgr B$.   Note that the section functor $\omega$ from Notation~\ref{section-functor} 
satisfies $\omega(\pi(N)) = \uHom_{\rqgr B}(\pi(B), \pi(N))$.  In particular, applying $\omega$ 
to our exact sequence gives an exact sequence 
\begin{equation}\label{3.8.1}
0 \to \omega(\pi(M_p)) \to \omega(\mc{F}) \to \omega(\pi(M_p)) \to \uExt^1_{\rqgr B}(\pi(B), \pi(M_p)) \to \dots
\end{equation}
Now 
\[
\uExt^1_{\rqgr B}(\pi(B), \pi(M_p)) = \bigoplus_m \Ext^1_{\rqgr B}(\pi(B), \pi(M_p[m])) = 
\bigoplus_m \Ext^1_{\coh E}(\mc{O}_E, \mc{O}_{\tau^m(p)}) 
= 0,
\]
using the equivalence of categories between $\rqgr B$ and $\coh E$ and the fact that sheaves with zero-dimensional support 
have vanishing higher cohomology.
However,
 $\omega(\pi(M_p))_{\geq 0} = M_p$ by construction and so 
\eqref{3.8.1} becomes the exact sequence $0 \to M_p \to \omega(\mc{F})_{\geq 0} \to M_p \to 0$. 
This is nonsplit since applying $\pi$ yields the original nonsplit extension.    
\end{proof}

\begin{proposition}\label{prop:Hompt} 
Let $p,q\in E$.  Then
\[ \tag{1}   \uHom_B(M_p, M_q) = \begin{cases} \kk[-j] & \text{if $p = \tau^j q$ for some $j \geq 0$} \\
0 & \text{otherwise,}  \end{cases} \]
while 
\[ \tag{2} {}\qquad \ \ \  \uExt_B^1(M_p, M_q) = \begin{cases} \kk[1] \oplus \kk[-j] & \text{if $p = \tau^j q$ 
for some $j \geq 0$} \\
\kk[1] & \text{otherwise.}  \end{cases} \]
\end{proposition}

\begin{proof} (1)
Using \eqref{eq:truncshift}, we have $\uHom_B(M_p,M_q)_j \neq 0$ $\iff $ $M_p[-j]  \cong (M_q)_{\geq j} \cong (M_{\tau^j q})[-j]$.
Clearly this happens if and only if $j \geq 0$ and $p = \tau^j q$; in particular it can happen for at most one value of $j$.

(2) Now consider   the  Ext groups.  Note that $\uExt^i_{\rqgr B} (M_p,M_q)= 
\lim_{n\to \infty} \uExt^i_{\rgr B}\bigl((M_p)_{\geq n},\,M_q)$, by \cite[Proposition~7.2]{AZ}.
 Thus the exact sequences 
 $0\to (M_p)_{\geq n} \to M_p\to M/(M_p)_{\geq n}\to 0$ induce the exact sequence 
\begin{multline}
\label{eq:extMpq}
 0 \too \uHom_B(M_p, M_q) \too  \uHom_{\rqgr B} (M_p,M_q )
  \overset{\alpha}{\too} \lim_{n \to \infty} \uExt^1_B(M_p/(M_p)_{\geq n}, M_q) \\ 
\overset{\beta}{\too} \uExt^1_B(M_p, M_q) \too \uExt^1_{\rqgr B} (M_p,M_q ) \too \dots\quad
\end{multline} 

We claim that  
$X=\lim_{n \to \infty} \uExt^1_B(M_p/(M_p)_{\geq n}, M_q)$ is  zero  in degrees $\geq 0$
while $\dim_{\kk} X_{-1} = 1$.
To see this,  apply  $\uHom_B( -, M_q)$  to  the exact sequence
\[ 0 \to \kk[-n] \too M_p/(M_p)_{\geq n+1} \too M_p/(M_p)_{\geq n} \too 0.\]
This gives
\[ 0\too  \uExt^1_B(M_p/(M_p)_{\geq n}, M_q) \too \uExt^1_B(M_p/(M_p)_{\geq n+1}, M_q) 
\ \overset{\phi}{\too}  \uExt^1_B(\kk[-n], M_q)\too \dots.\]
Since $ \uExt^1_B(\kk[-n], M_q) = \kk[n+1]$ by Lemma~\ref{sub-ext}, it follows  from this sequence
 and induction on $n$  that $\uExt^1_B(M_p/(M_p)_{\geq n}, M_q)$ lives 
in negative degrees, and that $\dim_{\kk} \uExt^1_B(M_p/(M_p)_{\geq n}, M_q)_{-1} = 1$ 
for all $n$. The  claim follows.

Next, write $M_p = B/I_p$, where $I_p = \bigoplus_{n \geq 0} H^0(E, \mc{I}_p \otimes \mc{M}_n)$.   
Since $\deg( \mc{I}_p \otimes \mc{M}) \geq 2$,   it   follows from Lemma~\ref{lem:basics for B}(1)
that $I_p  $ is generated in degree one.
Thus the graded free resolution of $M_p$ begins 
\[
 \dots \to \bigoplus_{i=1}^m B[-1] \too B \too M_p \too 0.
\]
Using this to calculate $\uExt^1_B(M_p, M_q)$  shows that $\uExt^1_B(M_p, M_q)_n = 0$ for $n \leq -2$.

Combining this with the conclusion of the previous paragraph shows
 that, in \eqref{eq:extMpq},  $\im \beta$ is concentrated entirely in degree $-1$, and is either $0$ or $\kk[1]$.
 In fact we can be more precise. By comparing Lemma~\ref{lem:extqgr} with Part (1), it follows that $\alpha=0$ except when $p=\tau^j(q)$ for some 
$j<0$ and in the latter case $\text{Im}(\alpha) = \kk[-j]$. 
We conclude that 
\begin{equation}\label{early-conc}
 \im\beta  =\begin{cases}   0 & \text{ \ if $p=\tau^{-1}(q)$}\\
 \kk[1] & \text{ \ otherwise.}\end{cases}\end{equation}

To complete the proof, we consider several cases.   
First, if $p,q$ lie on different orbits then Lemma~\ref{lem:extqgr} implies that 
  $ \uExt_{\rqgr B}(M_p, M_q) = 0$.   
Thus,   \eqref{early-conc} shows that 
 $  \uExt^1_B(M_p, M_q) =\im\beta \cong \kk[1]$.

Next, suppose that  $p = \tau^{-1}(q)$. Then,  $\beta = 0$ by \eqref{early-conc} 
and so  in this case   Lemma~\ref{lem:extqgr} implies that 
$\uExt^1_B(M_p, M_q)\hookrightarrow \uExt^1_{\rqgr B}(M_p, M_p) \cong \kk[1]$. 
Since $\kk[1] \hookrightarrow \uExt^1_B(M_p, M_q)$ by Lemma~\ref{sub-ext}(3), we conclude that
 $\uExt^1_B(M_p, M_q) \cong \kk[1]$, as required.

It remains to consider the case when $p = \tau^j(q)$ for some $-1\not=j \in \mb{Z}$.  
Comparing \eqref{early-conc} with Lemma~\ref{lem:extqgr}   gives an exact sequence 
\[
0 \too \kk[1] \too \uExt^1_B(M_p, M_q) \too \kk[-j] \too \dots.
\]
If $j \leq -2$, we have shown above  that $\uExt^1_B(M_p, M_q)_j = 0$. Hence $\uExt^1_B(M_p, M_q) \cong \kk[1]$.
If  $j \geq 0$, then we must show that $\uExt^1_B(M_p, M_q) \cong \kk[1] \oplus \kk[-j]$, for which it suffices to 
show that $0\not= \uExt^1_B(M_p, M_q)_j = \uExt^1_B(M_p, M_q[j])_0$.    From the exact sequence 
\[
0 = \Hom_{\rgr B}(M_p, M_q[j]/M_q[j]_{\geq 0}) \too \Ext^1_{\rgr B}(M_p, M_q[j]_{\geq 0}) \too \Ext^1_{\rgr B}(M_p, M_q[j])
\]
it suffices to show that  $\uExt^1_B(M_p, M_q[j]_{\geq 0})_0\not=0$ or, equivalently by \eqref{eq:truncshift}, that 
  $\uExt^1_B(M_p, M_p)_0 \neq 0$. In other 
words, we can reduce to the case $j = 0$, where the result is just Lemma~\ref{sub-ext}(4).    \end{proof}


\section{Elliptic algebras}\label{ELLIPTIC ALGEBRAS}

In this section, we define elliptic algebras, which are the main objects  of interest in this paper, and 
describe some of their more basic properties.

\begin{definition}\label{elliptic-defn}
A connected $\mb{N}$-graded algebra $R$ is called an \emph{elliptic algebra} if there 
is a central nonzerodivisor $ g \in R_1$ such that 
$R/gR \cong B(E, \mc{M}, \tau)$ for some elliptic curve $E$, invertible sheaf $\mc{M}$, and infinite order automorphism 
$\tau$.  We call $\deg \mc{M}$ the \emph{degree} of the elliptic algebra.  
{\it In this paper, we will always 
assume that an elliptic algebra has degree at least $3$ unless otherwise stated. }  
\end{definition}

\begin{example}\label{sklyanin-defn} Some of the most important elliptic algebras arise from the (quadratic) Sklyanin algebra
$$S=Skl(a,b,c)=\kk\{x_1,x_2,x_3\}/(ax_ix_{i+1}+bx_{i+1}x_i+cx_{i+2}^2 : i\in \ZZ_3),$$
where $(a,b,c)\in \mathbb{P}^2\smallsetminus \mathcal{S}$ for a (known) finite set $\mathcal S$.
Here, $S$ contains a canonical central element $g\in S_3$ such that $S/gS\cong B(E,\mathcal{L},\sigma)$ for an elliptic curve $E$. 
In this paper we restrict attention to the case when   $|\sigma|=\infty$ since    the 3-Veronese ring $T=S^{(3)}$ is then  an elliptic algebra, with 
$T/gT \cong B(E, \mc{M}, \tau)$  for  $\mc{M} = \mc{L} \otimes \mc{L}^{\sigma} \otimes \mc{L}^{\sigma^2}$ and $\tau = \sigma^3$.

Another, related example of an elliptic algebra can be obtained by taking the fourth Veronese ring $T'=(S')^{(4)}$ of the cubic 
Sklyanin algebra $S'$
from \cite{ATV1990}. More generally, there are  the (second Veroneses of the) quadrics  constructed in \cite{VdB3}. As discussed in
 \cite[Example~8.5]{SV}, these algebras can all be written as factors of a certain 4-dimensional Sklyanin algebra, although as they will not be needed 
explicitly in this paper we will omit the definitions.
\end{example}

An elliptic algebra $R$ automatically has a number of  good properties, basically because the same properties hold for 
the factor ring $B = R/Rg$.  
Before stating the result we need one more definition. Given a noetherian cg $\kk$-algebra $A$, regard  $\kk=A/A_{\geq 1}$
 as a right $A$-module. Then $A$ satisfies \emph{the  
Artin-Zhang $\chi$-condition} (on  the right)  provided $\dim_{\kk}\uExt^j_A(\kk,M)<\infty$ for all $M\in \rgr R$ and all $j \geq 0$.

\begin{proposition}\label{prop:elliptic}
Let $R$ be an elliptic algebra with $B=R/gR$.  Then both $R$ and $B$ are noetherian domains   generated in degree $1$. 
 In addition, $R$ and $B$ are  Auslander-Gorenstein, CM,   AS-Gorenstein, and   satisfy the Artin-Zhang $\chi$-condition.
\end{proposition}
\begin{proof}  By  Lemma~\ref{lem:basics for B},  $R/gR$ and hence $R$ are generated in degree one.  
Now the noetherian, Auslander-Gorenstein and CM properties as well as the $\chi$ condition hold for 
$B$ by \cite[Lemma~2.2]{R-Sklyanin} and for $R$ by \cite[Theorem~6.3]{R-Sklyanin}.  The proofs of these results also
easily imply that $\GKdim(B) = 2$ and $\GKdim(R) = 3$, so that by  \cite[Theorem~6.3]{Lev1992}, 
both $B$ and $R$ are also AS-Gorenstein.
\end{proof}

\begin{notation}\label{div-defn}
Let $R$ be an elliptic algebra with  factor ring $B = R/gR$.  
For a graded $R$-module $M$ or, indeed a  $\kk[g]$-module $M$, its \emph{$g$-torsion submodule} is $t(M) = \{ m \in M : g^nm = 0\ \text{for }\ n \gg 1 \}$. Then  
$M$ is \emph{$g$-torsion} if $M = t(M)$ and \emph{$g$-torsionfree} if $t(M) = 0$.   
Write $  R_{(g)}$ for the homogeneous localisation of $R$ at the completely prime ideal $gR$; thus $R_{(g)} = R\mathcal{C}^{-1}$,
 for $\mathcal{C}$ the set of homogeneous elements in $R\smallsetminus gR$.   As in \cite[Notation~2.5]{RSSlong},  $R_{(g)}/g R_{(g)} \cong Q_{gr}(B) \cong \kk(E)[t, t^{-1}; \tau]$.
We say that a $\kk[g]$-submodule $M \subseteq R_{(g)}$ is \emph{$g$-divisible} 
\label{g-div}  when $M \cap g R_{(g)} = Mg$.  It is easy to see that $M$ 
is $g$-divisible if and only if $R_{(g)}/M$ is $g$-torsionfree.  The fact that $R/gR\cong B$ is a domain forces $R$ itself to be $g$-divisible. 
When $M$ is $g$-divisible, 
 we  can and will identify $M/Mg$ with 
  \[\overline{M} = (M + g R_{(g)})/ g R_{(g)}\subseteq  R_{(g)}/g R_{(g)}.\]
\end{notation}
 
The following  properties of graded modules over an elliptic algebra will be useful.
 
\begin{lemma}\label{lem:ideal}
Let $R$ be an elliptic algebra.
\begin{enumerate}
\item   If $M, N \subseteq R_{(g)}$ are  $g$-divisible graded 
$R$-submodules, then $\uHom_R(M,N) \subseteq R_{(g)}$ is also $g$-divisible.
\item If $M$ is a $g$-torsionfree finitely generated graded right $R$-module with $\GKdim M \leq 1$, then $I = \Ann_R(M)$ is a nonzero graded
ideal of $R$ with $\GKdim R/I \leq 1$.
\end{enumerate}
\end{lemma}
\begin{proof} (1)  See \cite[Lemma~2.12(2)]{RSSlong}. 

(2) As $M/Mg$ is finite dimensional,   $M^\circ = M[g^{-1}]_0$ is a finite dimensional module over $R^\circ=R[g^{-1}]_0$. 
Hence $J=\Ann_{R^\circ}(M^\circ)\not=0$.   Therefore,  $0\not=\widehat{J}=\bigcup_{n} \{a\in R_n  : ag^{-n}\in J\}$ and certainly 
  $\widehat{J}\subseteq \Ann_RM$; whence $I=\Ann_RM\not=0.$  As $M$ is $g$-torsionfree, 
$I\not\subseteq gR$. Hence, by   \cite[Lemma~2.15(3)]{RSSlong},
$\GKdim R/I \leq 1$. 
\end{proof}

\begin{lemma} \label{lem:reflexive}
Let $R$ be an elliptic algebra with $R/Rg = B$.
\begin{enumerate}
\item  Let $J\in \rgr R$ with $J \subseteq Q = Q_{gr}(R)$.  Then $J^{**}$ is the unique largest $R$-submodule 
of $Q$ such that $\GKdim J^{**}/J \leq 1$.   In particular, $J$ is reflexive if and only if $Q_{gr}(R)/J$ is $2$-pure.  
  
\item  Let    $J \subseteq R_{(g)}$ be a finitely generated, $g$-divisible graded right $R$-submodule.  If  
$\overline{J}=J/Jg$ is saturated as a $B$-module, then $J$ is reflexive as a right $R$-module.
\end{enumerate}
\end{lemma}

\begin{proof}
(1) This follows from the CM property; see, for example, \cite[(4.6.6) and Remark~5.8(4)]{Lev1992}.

(2) We know that $J$ is reflexive if and only if $Q_{gr}(R)/J$ is $2$-pure.   This is equivalent to $R_{(g)}/J$ 
being $2$-pure, since $Q_{gr}(R)/R_{(g)} = \bigcup_{n \geq 1} g^{-n}R_{(g)}/R_{(g)}$   is $2$-pure.

  If  $J$ is   not reflexive, 
  there exists a finitely generated module  $J \subsetneq N \subseteq R_{(g)}$ with   $\GKdim(N/J) \leq 1$.  By 
  Lemma~\ref{lem:ideal}(2) and the fact that $R_{(g)}/J$ is $g$-torsionfree, 
  $NI \subseteq J$ for some graded ideal $I$ of $R$ with $\GKdim R/I \leq 1$.
  Let $\wh{N} = \{ y \in R_{(g)} : yg^n \in N,\ \text{some}\ n \geq 0 \}$.  Then  $\wh{N}$ is   $g$-divisible, 
with $\wh{N}I \subseteq J$.   Since $\wh{N} \subseteq \uHom_R(I,J)$, clearly $\wh{N}$ is left bounded.
Since $J \subsetneq \wh{N}$ and both are left bounded and $g$-divisible, 
we must have $\overline{J} \subsetneq \overline{\wh{N}}$; otherwise  $J/Jg = \wh{N}/\wh{N}g$ and 
the graded Nakayama lemma would  imply that $J=\wh{N}$.
  Moreover, $\overline{\wh{N}} \overline{I} \subseteq \overline{J}$.
But $\dim_\kk  B/\overline{I} < \infty$, since all nonzero ideals of $B$ have finite codimension (see, for example,
\cite[Lemma~4.4]{AS}).
Hence $\dim_{\kk}  \overline{\wh{N}}/ \overline{J}<\infty$, showing $\overline{J}$ is not saturated, a contradiction.
\end{proof}

The next few lemmas provide useful homological properties for modules over  an elliptic algebra.
First however we prove an elementary result that will be used several times.

\begin{lemma}\label{lem:barhom1}
Let $H$ be a locally finite, left bounded, graded $\kk[g]$-module. If the multiplication map $\timesg$ has a finite dimensional kernel on $H$, 
then the $g$-torsion submodule of $H$ is also finite dimensional.
\end{lemma}

\begin{proof}
By hypothesis,  $\Ann_H(g)$ is finite-dimensional, say contained in degrees $\leq d$.  Now if $0 \neq x \in H$ is $g$-torsion, 
then pick  $n \geq 1$  minimal with $g^n x = 0$; thus  $0 \neq g^{n-1}x \in \Ann_H(g)$.  It follows that $\deg x\leq d - n+1$.
In particular, the $g$-torsion submodule of $H$ is entirely contained in degrees $\leq d$,  and so is finite dimensional as $H$ is 
left bounded and locally finite.  \end{proof}

\begin{lemma}
\label{lem:extRextB}
Let $R$ be an elliptic algebra, with $B = R/gR$.  Suppose that $M \in \rGr R$ is $g$-torsionfree and 
that $N \in \rGr B$.
Then,   for all $i \geq 0$, one has 
 $\uExt^i_R(M,N) \cong \uExt^i_B(M/Mg, N)$  and 
\[  \uExt^i_{\rQgr R}(\pi(M), \pi(N)) \cong \uExt^i_{\rQgr B}(\pi(M/Mg), \pi(N)).\]
\end{lemma}
\begin{proof}
Both parts are essentially the same easy exercise; cf. \cite[Proposition~5.1.2(1)]{VdB-blowups}.  \end{proof}
 
\begin{lemma}\label{lem:barhom}
Let $R$ be an elliptic algebra with $R/Rg = B = B(E, \mc{M}, \tau)$. Let $I$ and $J$ be non-zero $g$-divisible, reflexive finitely generated graded right $R$-submodules 
of $R_{(g)}$.  Then 
\begin{enumerate}
\item The natural inclusion $ \overline{\uHom_R(I,J)} \subseteq \uHom_{B}(\overline{I}, \overline{J})$ has 
a finite-dimensional cokernel.
\item  The $g$-torsion subspace of $\uExt^1_R(I,J)$ is finite-dimensional over $\kk$.
\end{enumerate}
\end{lemma}
\begin{proof}
(1) 
The proof is a variant of \cite[Prop. 6.12]{RSSlong}.   First,  replacing $I$ and $J$ by $xI$ and $yJ$, for some    
homogeneous  elements $x, y \in R \setminus gR$, we can assume without loss that 
  $I, J \subseteq R$.    Note that   $R/I$ and $ R/J$  are  
   $g$-torsionfree modules  and hence,  by Lemma~\ref{lem:reflexive},    are either  2-pure or $0$.

By Lemma~\ref{lem:extRextB}  we may 
 identify $\uExt^i_R(R/I,\overline{J}) = \uExt^i_B(B/\overline{I}, \overline{J})$. Thus, applying $\uHom_R(R/I,-)$ to the sequence 
 $0\to J\overset{ \timesg }{\too}  J\to \overline{J}\to 0$ gives the exact sequence
\begin{gather}
\label{eq:barhom}
\uHom_B(B/\overline{I}, \overline{J}) \too \uExt^1_R(R/I, J)[-1] \overset{\timesg}{\too} \uExt^1_R(R/I, J) \overset{\beta}{\too} \uExt^1_B(B/\overline{I}, \overline{J})  \\
\nonumber \overset{\alpha}{\too} \uExt^2_R(R/I, J)[-1] \overset{\timesg}{\too} \uExt^2_R(R/I, J) \too \uExt^2_B(B/\overline{I}, \overline{J}) \too \dots 
\end{gather}
Moreover, $\uHom_B(B/\overline{I}, \overline{J}) = 0$ since $B$ is a domain.

We claim that $\dim_{\kk} \uExt^2_B(B/\overline{I}, \overline{J}) < \infty$.  Indeed, since $B$ is a domain with $\GKdim(B) = 2$, \cite[Proposition~5.1(e)]{KL}   
 implies that     $\GKdim B/\overline{I} \leq 1$   and so $B/\overline{I}$
has a finite filtration  by point modules and finite-dimensional modules.  By Proposition~\ref{prop:elliptic}
$B$ satisfies the Artin-Zhang $\chi$-condition and so $\dim_{\kk} \uExt^2_B(\kk, \overline{J}) < \infty$.
Thus in order to prove the claim,  using the usual long exact sequences in cohomology, it suffices  to show that 
$\dim_{\kk} \uExt^2_B(M_p, \overline{J}) < \infty$ for a point module $M_p$.  
Consider the exact sequence
\[
\dots \to \uExt^1_{B}(M_p, B/\overline{J}) \to 
\uExt^2_{B}(M_p, \overline{J}) \to \uExt^2_{B}(M_p,  B) \to \dots
\]
Here,  $\uExt^2_B(M_p, B) = 0$  by Lemma~\ref{lem:Bdual}. Also $B/\overline{J}$ is again
filtered by point modules and finite-dimensional modules.  Obviously $\dim_{\kk} \uExt^1_B(M_p, \kk) < \infty$, while 
 $\dim_{\kk} \uExt^1_B(M_p, M_q) < \infty$ 
for a point module $M_q$ by Proposition~\ref{prop:Hompt}.    Thus $\dim_{\kk} \uExt^1_{B}(M_p, B/\overline{J}) < \infty$ 
and  so $\dim_{\kk} \uExt^2_{B}(M_p, \overline{J}) < \infty$, as  claimed.

Equation \ref{eq:barhom} now shows that if $N = \uExt^2_R(R/I, J)$, then the map 
$N_{n-1} \overset{\timesg}{\too} N_n$ is surjective for $n \gg 0$.  Since $N$ is left bounded and locally finite (see Notation~\ref{loc-fin})  this forces 
$\dim_{\kk} N_n = d$ for some constant $d$ and any $n \gg 0$. Hence 
$N_{n-1} \overset{\timesg}{\too} N_n$ 
is an isomorphism for $n \gg 0$ and so $\text{Coker}(\beta) \cong \text{Im}(\alpha)$ is finite-dimensional.       Identifying $\uExt^1_R(R/I, J) \cong \uHom_R(I,J)/J$ and 
$\uExt^1_{B}(B/\overline{I}, \overline{J}) \cong \uHom_B(\overline{I}, \overline{J})/\overline{J}$, 
this means that the natural map $ \uHom_R(I,J)/J \to \uHom_B(\overline{I}, \overline{J})/\overline{J}$
  has a finite-dimensional cokernel.
As in the proof of \cite[Prop 6.12]{RSSlong}, it 
easily follows   that the natural map
$ \uHom_R(I,J) \to \uHom_B(\overline{I}, \overline{J})  $ also   has a finite-dimensional cokernel.

(2)
   From the exact sequence
\[
0 \too \uHom_R(I, J)[-1] \overset{\timesg}{\too} \uHom_R(I, J) \too \uHom_B(\overline{I}, \overline{J}) \too \uExt^1_R(I,J)[-1] \overset{\timesg}{\too} \uExt^1_R(I,J)   
\]
and part (1), the  map  $\timesg$  on $H= \uExt^1_R(I,J)$ has finite-dimensional kernel.  Now apply   Lemma~\ref{lem:barhom1}.  \end{proof}


\section{Line modules}\label{LINES}

The main aim of this paper is to create an algebraic analogue of contracting lines of self intersection $-1$. To this end,
 in this section we discuss line modules---the appropriate analogues of lines---while in the next section we discuss their 
 intersection theory.  Throughout the section, we fix an   elliptic algebra $R$ of degree $\geq 3$ 
with $R/Rg = B = B(E, \mc{M}, \tau)$.

\begin{definition}\label{line-defn}
A (\emph{right}) \emph{line module} over the elliptic algebra $R$ is a cyclic   graded  $R$-module $L\in\rgr R$ with 
Hilbert series $\hilb L = 1/(1-s)^2$.  Because $\Hom_{\rgr R}(R, L) = \kk$ for such a module $L$,  there is a unique 
right ideal $J$ of $R$  with $L \cong R/J$.  We refer to $J$ as the {\em line ideal of $L$}.
\end{definition}

\begin{lemma}\label{lem:basics}
Let $L$ be a right  line module over the elliptic algebra $R$.  Then  
$L$ is $g$-torsionfree and $2$-critical, and $L/Lg$  is a point module.  
\end{lemma}

 \begin{proof} 
This follows from \cite[Lemma~8.9]{R-Sklyanin}.
\end{proof}

 Recall from Lemma~\ref{lem:basics for B} the equivalence of categories $\Qcoh E \to \rQgr B$.  Since the simple 
objects in $\Qcoh E$ are the skyscraper sheaves $\mc{O}_p$ for points $p \in E$, following Definition~\ref{point-defn} the 
simple objects in $\rqgr B$ are the images $\pi(M_p)$ of the point modules $M_p$, and so these 
are also parametrised by closed  points $p\in E$. By a slight abuse of notation we will often write $\pi(M_p)=\mc{O}_p$ to emphasise the correspondence.

\begin{definition}\label{divisor-defn}
Let $M$ be a right   line module or, more generally, a finitely generated $g$-torsionfree right $R$-module with $\GKdim M = 2$.  Then
    $\GKdim M/Mg = 1$, and so  $\pi(M/Mg)\in \rqgr B$ has finite length.  
Thus $\pi(M/Mg)$ has a filtration with simple factors 
$\mc{O}_{p_1}, \dots, \mc{O}_{p_n}$, for some $p_i\in E$.  
We  define  the \emph{divisor} of   $M$ to be $\Div M = p_1 + \dots + p_n$.  
In particular, if $L$ is a line module then $L/Lg \cong M_p$ and $\Div L = p$  for some point $p\in E$.
The analogous notation will be used for left modules.  \end{definition}

\begin{lemma}
\label{lem:exths}   
Let $M\in\rgr R$ be  $g$-torsionfree  and assume   that $M/Mg$ 
has a filtration with shifted point module factors $\{ M_{p_i}[m_i] : 1 \leq i \leq d \}$.  
\begin{enumerate}
\item $M$ is $2$-pure and CM, with  $\Div M = \sum \tau^{m_i}(p_i)$ and 
$\hilb M = \sum_{i=1}^d s^{-m_i}/(1-s)^2$.
\item Let $N = \uExt^1_R(M, R)$.  Then $N\in R\lgr$  and is  $g$-torsionfree, 
$2$-pure, and CM. 
Moreover, $N/gN$ has a finite filtration with shifted left point module factors 
$\{ M^{\ell}_{\tau^{-2}(p_i)}[-m_i - 1]  :  1 \leq i \leq d \}$.
In particular, $\Div N = \sum \tau^{m_i-1}(p_i) $, and $\hilb N = \sum_{i=1}^d s^{m_i+1}/(1-s)^2$.
\end{enumerate}
\end{lemma}

\begin{proof} 
(1)  If   $M$ is not $2$-pure, then it has a submodule $H$ with $\GKdim H<2$.
Necessarily $\GKdim H = 1$ since finite-dimensional modules are $g$-torsion. Also, 
as  $M$ is $g$-torsionfree,    $H(n)=\{f \in M: fg^n\in H\}$  satisfies $\GKdim H(n)=1$ for any $n$. Thus,
after  replacing $H$ by some such $H(n)$ we can assume that $H\not\subseteq Mg$. In this case, 
$\dim_\kk H/Hg<\infty$ and hence  $(H+Mg)/Mg \cong H/H\cap Mg$ is a   nonzero,  finite-dimensional submodule of $M/Mg$, contrary to assumption.  Thus $M$ is $2$-pure.

 Recall from  Lemma~\ref{lem:extRextB} that  $\uExt^j_R(M, B) \cong \uExt^j_B(M/Mg, B)$. 
Thus, applying $\uHom_R(M,-)$ to  the exact sequence $0 \to R[-1] \stackrel{g}{\to} R \to B \to 0$, gives the long exact sequence 
\begin{equation}\label{eq:exths} \cdots\too
\uExt^{j-1}_B(M/Mg, B) \too  \uExt^j_R(M, R)[-1] \ \overset{\timesg}{\too} \ \uExt^j_R(M, R) \too
\uExt^j_B(M/Mg, B)  \too  \cdots\end{equation}

Given  a point module $M_p$, then  Lemma~\ref{lem:Bdual} implies that 
$
\uExt^j_B(M_p, B)=0 $ for $j\not=1$. 
Since $M/Mg$ is filtered by point modules, standard long exact sequences show  that 
$\uExt^j_B(M/Mg, B) = 0$ for $j \neq 1$.   
In particular if $j \not=1$, then  
the map $\timesg $ in \eqref{eq:exths}  is surjective.   Since 
$\uExt^j_R(M,R)$ is left bounded, this can only happen   if $\uExt^j_R(M,R) = 0$.  Thus $M$ is Cohen-Macaulay.   

 The computations of 
 $\hilb M$ and   $\Div M$ follow routinely from the hypotheses, using \eqref{eq:truncshift}.

(2)   By the proof of part (1), the    sequence \eqref{eq:exths} collapses to give 
the short exact sequence 
\[ \ 
0 \too  \uExt^1_R(M, R)[-1]   \ \overset{\timesg}{\too} \
 \uExt^1_R(M, R) \too \uExt^1_B(M/Mg, B) \too 0. 
\]
Hence $N =  \uExt^1_R(M, R)$ is a $g$-torsionfree left $R$-module with $ N/gN \cong \uExt^1_B(M/Mg, B)$.  Clearly $N$ is finitely generated.   
Given  a point module $M_p$, then  Lemma~\ref{lem:Bdual} implies that 
$
\uExt^1_B(M_p, B) \cong   M^{\ell}_{\tau^{-2}(p)}[-1] .$
 Since $M/Mg$ is filtered by the $\{ M_{p_i}[m_i] \}$, it follows 
  that $N/gN$  is filtered by the left point modules $\{ M^{\ell}_{\tau^{-2}(p_i)}[-m_i -1] \}$.   
     The values of $\hilb N$ and $\Div N$
follow, as they did for $M$,  but using \eqref{leftpoint} in place of \eqref{eq:truncshift}.
Similarly, $N$ is $2$-pure and CM by the  arguments  from part (1).  
\end{proof}

We now consider some properties of line modules. 

\begin{lemma}\label{lem:red0}
Let $L, L'$ be right line modules over an elliptic algebra $R$, with $\Div L = p$ and $\Div L' = {p'}$. 
\begin{enumerate}
\item If $p \neq \tau^j(p')$ for any $j \geq 0$, then $\uHom_R(L,L') = 0$.
\item If $p = \tau^j(p')$ for some $j \geq 0$, then either $\uHom_R(L,L') = 0$ or else $\hilb \uHom_R(L, L') = s^j/(1-s)$. 
\item $\uEnd_R(L) \cong \kk[g]$; in particular $\hilb \uEnd_R(L) = 1/(1-s)$.
\end{enumerate}
\end{lemma}

\begin{proof}
By definition, $L/Lg = M_p$ and $L'/L'g = M_{p'}$.  
By Lemma~\ref{lem:extRextB},   $\uExt^i_R(L, M_{p'}) \cong \uExt^i_B(M_p, M_{p'})$
for all $i \geq 0$.  
Applying $\uHom_R( L, -)$ to the short exact sequence $0 \to L'g \to L' \to M_{p'} \to 0$ gives 
\begin{equation}\label{eq:red0}
0 \too \uHom_R(L,L')[-1] \overset{\timesg}{\too} \uHom_R(L, L') \overset{\gamma}{\too}   
\uHom_B(M_p, M_{p'}) \too \cdots .
\end{equation}

(1)  In this case,   Proposition~\ref{prop:Hompt} implies that  $\uHom_B(M_p, M_{p'}) = 0$ and so $\timesg$ is 
surjective in \eqref{eq:red0}. Since   $\uHom_R(L,L')$ is left bounded, this forces $\uHom_R(L,L') = 0$.

(2) Here, Proposition~\ref{prop:Hompt} implies that $\uHom_B(M_p, M_{p'}) = \kk[-j]$.  Thus if $\uHom_R(L,L') \neq 0$ 
then $\gamma$  is surjective  in \eqref{eq:red0}. As $\uHom_R(L,L')$ is left bounded, this forces $\hilb \uHom_R(L, L') = s^j(1-s)^{-1}$.

(3) As there is a natural graded inclusion $\kk[g] \hra \underline{\End}_R(L)$, part (2) implies that  $\hilb \uHom_R(L, L) = (1-s)^{-1}$ and hence that $\uEnd_R(L) = \kk[g]$.
\end{proof}

\begin{lemma}\label{lem:Rdual}  Let $R$ be an elliptic algebra, with $R/Rg = B = B(E, \mc{M}, \tau)$.
Let $L$ and $L'$ be right line modules over $R$.  
\begin{enumerate}

\item    $L$ is CM, and 
$L^{\vee} = \underline{Ext}^1_R(L, R)[1]$ is a left   line module.  
\item Under the natural morphism, $L=L^{\vee \vee}$. 
 Further, if $J$ is the line ideal of $L$, then 
$J=J_1R$ is $g$-divisible, CM  and reflexive, while $\overline{J}=J/Jg$ is saturated.
\item Up to isomorphism,  there is a unique non-split exact sequence $ 0 \to R \to M \to L[-1] \to 0$.  
Explicitly, if   $L^{\vee} = R/J^{\vee}$, then $M = (J^{\vee})^*$.  This $M$ is $g$-divisible, CM and reflexive.
\item For $j=0,1$   one has $\uExt^j_R(L', L) \cong \uExt^j_R(L^{\vee}, (L')^{\vee})$ as graded vector spaces.   
\item   $\uExt^1_R(\kk, L) = 0$.
\end{enumerate}
\end{lemma}
 
\begin{proof}
(1)  Write $L/Lg=M_p$ for some $p\in E$. Then 
 Lemma~\ref{lem:exths}   shows that  
 $L$ is CM, and that $N = \uExt^1_R(M_p,R)$
  has Hilbert series $s/(1-s)^2$, with 
$N/gN \cong M^{\ell}_{\tau^{-2}(p)}$.  In particular, since $N/gN$ is cyclic, $N$ is cyclic by the graded Nakayama lemma, 
and since $L^{\vee} = N[1]$ has the Hilbert series of a line module, it is a left line module.

(2) Mimicking the notation from  Lemma~\ref{lem:Bdual},  set 
$E^{ij}(N) = \uExt^i_R(\uExt_R^j(N, R), R)$ for a graded $R$-module $N$.
We first note that the  natural morphism, $L\to E^{11}(L) = L^{\vee \vee}$
 is obtained as follows: applying $\uHom(-, R)$ to $0 \to R \to J^* \to  \uExt^i_R(L, R) \to 0$
gives the exact sequence  $0 \to J^{**} \to R \to E^{11}(L)$.
Since $J^{**} \supseteq J$ this induces a homomorphism   from $L = R/J $ to $ R/J^{**} \subseteq E^{11}(L)$.

By  Lemma~\ref{lem:exths}(2)  $L^{\vee}$ is CM. Thus  the Gorenstein spectral sequence 
\eqref{spectral-eq} collapses to show that the natural  morphism $L\to   L^{\vee \vee}$
is an isomorphism.    Since $L = R/J$ is $g$-torsionfree, and $R$ is $g$-divisible, $J$ must be $g$-divisible.
Since $\overline{R}/\overline{J} \cong M_p$, necessarily $\overline{J} = \bigoplus_{n \geq 0} H^0(E, \mc{M}_n(-p))$.  
In particular, $\overline{J}$ is saturated, and so $J$ is reflexive by Lemma~\ref{lem:reflexive}(2).   
Similarly, $\overline{J}$ is generated in degree one by Lemma~\ref{lem:basics for B}(1) and hence $J$ is generated in degree one by the graded Nakayama lemma.
 Since $L = R/J$ is CM 
by part (1),  it follows routinely  that $\uExt^i_R(J,R) = 0$ for $i \geq 1$ and hence that $J$ is CM.

 (3) We have an exact sequence
  $0 \to R \to (J^{\vee})^* \to L^{\vee \vee}[-1] \to 0.$  
Since $(J^{\vee})^*$ is contained in the graded quotient ring $Q_{\operatorname{gr}}(R)$, the inclusion $R \to (J^{\vee})^*$ is 
 essential; in particular, the exact sequence is nonsplit.
 Now  $L^{\vee \vee} = L$ by part (2) and  $\Ext^1_{\rgr R}(L[-1], R) = \kk$, by part (1).  
Thus up to isomorphism there is a unique nonsplit degree 1 extension $M$ of $R$ by $L$, 
and it is given by $M = (J^{\vee})^*$.  Since $J^{\vee}$ is $g$-divisible by part (2),   $M = \uHom_R(J^{\vee}, R)$ is $g$-divisible by Lemma~\ref{lem:ideal}.
Finally, as $J^\vee$ is CM by the left-sided analogue of part (2),  
the  spectral sequence \eqref{spectral-eq} collapses for 
$J^\vee$ and shows that $\uExt^i_R(M,R)= E^{i0}(J^\vee)= 0 $ for $i>0$. In other words, $M$ is CM.
 
(4) We begin with $\uExt^1$.  
 Define a map $(-)^\vee:  \uExt^1_R(L',L) \to \uExt^1_R(L^\vee,(L')^\vee)$ as follows: let
\[ \mathsf{E}:\ \  0\too  L \ \overset{\alpha}{\too}\ M\ \overset{\beta}{\too}\ L'[j] \too 0\]
be an exact sequence corresponding to an element of $\uExt^1_R(L',L)_{-j}$.
Applying $\uHom_R(-,R)$ to $\mathsf{E}$ and using the fact that  $L$ and  $L'$ are CM gives
the dual extension
\[ \mathsf{E}^\vee:\ \  0\too   (L')^\vee[-j] \ \overset{\beta^\vee}{\too}  \uExt^1_R(M,R)[1] 
\ \overset{\alpha^\vee}{\too}  L^\vee \too 0.\]
We leave to the reader the verification using Baer sums \cite[Theorem~III.2.1]{MacLane} that $(-)^{\vee}$ is a linear transformation.

The double dual of $\mathsf{E}$ is
\[ \mathsf{E}^{\vee \vee}:\ \  0\too  L \ \overset{\alpha^{\vee \vee}}{\too}\ E^{11}(M)\ \overset{\beta^{\vee \vee}}{\too}\ L'[j] \too 0.\]

Since $L$ and $L'$ are CM, so is $M$ and hence  \eqref{spectral-eq} implies that  $M \cong E^{11}(M)$.
The functoriality of \eqref{spectral-eq}   
 then ensures that $\alpha^{\vee\vee} = \alpha$ and $\beta^{\vee\vee} = \beta$, whence $\mathsf{E}^{\vee\vee}\cong \mathsf{E}$.

Thus $(-)^{\vee}$ induces linear maps $f: \uExt^1_R(L', L) \to \uExt^1_R(L^{\vee}, (L')^{\vee})$ and 
$g: \uExt^1_R(L^{\vee}, (L')^{\vee}) \to \uExt^1_R(L', L)$ such that $g \circ f$ is the identity.   The same argument 
starting with the sequence $\mathsf{E}^{\vee}$ shows that $f \circ g$ is the identity.  Thus 
  $(-)^\vee$ is   an isomorphism of graded vector spaces.
   
In order to prove the result for  $\uHom$, suppose that $0\not=f\in \uHom_R(L',L)$. Then  $L$ and $L'$ are   
 GK-critical by Lemma~\ref{lem:basics}  and hence $f$ is an injection. Now applying $\uExt^1(-,R) $  
produces a  monomorphism $f^\vee \in \uHom_R(L^\vee, (L')^\vee)$ and  so the map $f\mapsto f^\vee$ defines  an injection 
$\uHom(L',L)\hookrightarrow \uHom(L^\vee, (L')^\vee)$. The fact that this is an isomorphism then follows by applying 
parts~(1, 2).

(5) The exact sequence $ 0 \to R\to (J^\vee)^* \to L[-1]\to 0$ from part (3) induces 
the exact sequence
\[ \uExt_R^1(\kk, (J^\vee)^*) \to \uExt_R^1(\kk, L)[-1] \to \uExt_R^2(\kk, R). \]
As $R$ is AS-Gorenstein of dimension 3 by Proposition~\ref{prop:elliptic}, the last term is zero.   
The first term is 0 by reflexivity of $(J^\vee)^* $, and the result follows.
\end{proof}


\section{Intersection theory}\label{INTERSECTION THEORY}

There is a general notion of intersection product on a noncommutative scheme, due to Mori and Smith \cite{MS},
that reduces to the usual definition for a commutative scheme but is more convenient when working in a noncommutative 
setting.  In this section we  give several alternative formul\ae\ for
  the intersection  product of line modules over  elliptic algebras. One drawback of the definition is that it is not always 
  defined for schemes of infinite homological dimension, so we also give a variant that is always defined.

\begin{definition}\label{intersect-ms}   
Let $R$ be a  connected noetherian $\mb{N}$-graded algebra. Then the
\emph{intersection number} of $\sM,\sN\in \rqgr R$  is defined  to be 
$$(\sM\dotms \sN) =\sum (-1)^{n+1}\dim_{\kk} \Ext^n_{\rqgr R}(\sM,\,\sN),$$
where we declare that the intersection is \emph{undefined} if   infinitely many terms are non-zero. 

Given $M,N\in \rgr R$, we define $(M\dotms N) = (\pi(M)\dotms \pi(N))$, although as above  
since the category will be clear from the context we will usually write $M$ for the image  in $\rqgr R$ of $M\in \rgr R$.  
\end{definition}

 \begin{notation}\label{grk-defn}
Given an   elliptic algebra $R$,  set $R^{\circ} = R[g^{-1}]_0$. Similarly, for   $M \in \rGr R$, set  
$M^{\circ} = M[g^{-1}]_0\in R^\circ\text{-Mod}$.  
   Since $R[g^{-1}]$ is strongly graded,
  $R[g^{-1}] \cong R^{\circ} \otimes_\kk \kk[g, g^{-1}]$, and  
there is an equivalence of categories $F: \rMod R^{\circ} \to \rGr R[g^{-1}]$, given by 
$F(N) = N \otimes_\kk \kk[g, g^{-1}]$. 

Finally, write $\grk M=\grk_{\kk[g]}M$ for the torsionfree rank of a $\kk[g]$-module $M$.
\end{notation}

\begin{lemma}\label{lem0.1} Let $R$ be an elliptic algebra, with
 $M,N\in\rgr R$.
 Then, for each $m\geq 0$,  $\uExt_{\rqgr R}^m(M,\,N)$ is a right $\kk[g]$-module with 
 $$\uExt_{\rqgr R}^m(M,\,N)\otimes_{\kk[g]}\kk[g,g^{-1}]  
\ \cong\ \uExt_{ R}^m(M,N)\otimes_{\kk[g]}\kk[g,g^{-1}]\ \cong\  \Ext_{R^\circ}^m(M^\circ,\,N^\circ) \otimes_\kk \kk[g, g^{-1}].$$
In particular, 
$
\grk \uExt^m_{\rqgr R}(M,N) = \grk \uExt^m_{R}(M,N) = \dim_\kk \Ext_{R^{\circ}}^m(M^{\circ}, N^{\circ}).
$
\end{lemma}

\begin{proof} The first assertion follows from the fact that $g$ is central.

By  Proposition~\ref{prop:elliptic} and 
\cite[Corollary~7.3(2)]{AZ}, there is a  map $\theta: \uExt^m_R(M,N) \to \uExt^m_{\rqgr R}(M,N)$ with 
right bounded kernel and cokernel.  As $g$ is central,  $\theta$ is a map of $\kk[g]$-modules, and  so the 
kernel and cokernel  of $\theta$ are $g$-torsion. This proves the  first isomorphism  in the display.

Next, using that $\Ext$ commutes with central localisation, we calculate that 
\[ 
\uExt_{R}^m(M,\, N)\otimes_{\kk[g]}\kk[g,g^{-1}]  \cong
\uExt_{R[g^{-1}]}^m(M[g^{-1}],\,N[g^{-1}]) \cong 
\Ext_{R^\circ}^m(M^\circ,\,N^\circ) \otimes_\kk \kk[g, g^{-1}],
 \]
where the final isomorphism  uses 
the  equivalence of categories $\rgr R[g^{-1}] \simeq \text{mod-}R^\circ$.  This gives the second isomorphism
in the display, from which  
the final equation is an easy consequence.
\end{proof}

We now consider in more detail the homological properties of line modules over the elliptic algebra $R$.

\begin{lemma}\label{lem0.3}
Let $L$ and $L'$ be line modules over the elliptic algebra $R$, with point factors
 $L/Lg = M_p$ and $L'/L'g = M_{p'}$.  Then there is a long exact sequence 
\begin{gather}
\label{eq:longex}
0 \too \uHom_{\rqgr R}(L,L')[-1] \overset{\timesg}{\too} \uHom_{\rqgr R}(L,L') \too  \FF  \too
 \uExt^1_{\rqgr R}(L,L')[-1] 
  \\ 
\nonumber 
\qquad\qquad \overset{\timesg}{\too} \uExt^1_{\rqgr R}(L,L')  \too \FF \too \uExt^2_{\rqgr R}(L,L')[-1] \overset{\timesg}{\too} \uExt^2_{\rqgr R}(L,L') \too 0,
\end{gather}
where $\FF \cong \kk[-j]$ if $p = \tau^j(p')$ for some $j \in \mb{Z}$, and $\FF = 0$ if $p$ and $q$ lie on different orbits.

For $m\geq 3 $,  multiplication by $g$ induces isomorphisms  $\uExt^m_{\rqgr R}(L,L')[-1]\cong \uExt^m_{\rqgr R}(L,L')$.
 
\end{lemma}
 
\begin{proof}  Applying $\uHom_{\rqgr R}(L,-) $
to $0\too L'[-1]\overset{\timesg}{\too} L'\too M_{p'}\too 0$ gives the long exact sequence
\begin{equation}\label{eq:0.4}
\cdots\to \uExt_{\rqgr R}^{n-1}(L,\,M_{p'}) \to  \uExt^n_{\rqgr R}(L,L'g) \to \uExt_{\rqgr R}^n(L,L') \to \uExt_{\rqgr R}^n(L,\,M_{p'}) \to  \cdots 
\end{equation}
The lemma now follows by using Lemma~\ref{lem:extRextB} to identify 
$\uExt_{\rqgr R}^n(L,\,M_{p'}) = \uExt_{\rqgr B}^n(M_p, M_{p'})$
and then applying Lemma~\ref{lem:extqgr}. 
    \end{proof}

\begin{proposition}\label{prop:0.6}  Let $R$ be an elliptic algebra and let
 $L, L' \in\rgr R $ be line modules,  with $p = \Div L$, $p' = \Div L'$.  Assume that  $(L \dotms L')$ is defined.
 \begin{enumerate}
\item  
$(L\dotms L'[m]) = (L\dotms L')$  for all $m \in \mb{Z}$.

\item  $(L\dotms L' )  =\sum (-1)^{n+1} \grk \,\uExt^n_{\rqgr R}(L,L')  
= \sum (-1)^{n+1}\dim_{\kk}\Ext^n_{R^\circ}(L^\circ,(L')^\circ).
$
\end{enumerate}\end{proposition}

\begin{proof} 
(1)  Restrict the morphisms in  Lemma~\ref{lem0.3}  to some degree $j$ and 
 take the alternating sum of the dimensions of the  resulting vector spaces in these equations. 
 Since the contributions from $\FF$ cancel, this gives 
\[
\sum_{n \geq 0}  (-1)^{n+1} \dim_\kk \uExt^n_{\rqgr R}(L,L'[-1])_j = \sum_{n \geq 0}  (-1)^{n+1} \dim_\kk \uExt^n_{\rqgr R}(L,L')_j.
\]
Thus, by \eqref{elementary-shift}, 
$(L\dotms L'[j-1]) = (L\dotms L'[j])$ and, by induction,
$(L\dotms L'[m]) = (L\dotms L')$ for all $m \in \mb{Z}$.

(2)  By  Lemma~\ref{lem0.1}, we only need to prove that $(L\dotms L')   =\sum (-1)^{n+1} \grk \,\uExt^n_{\rqgr R}(L,L')$.  Suppose first that  $p = \tau^j(p')$ for some 
$j \in \mb{Z}$. Then $(L\dotms  L'[j]) = (L\dotms L')$ by part (1), and 
\[
\sum (-1)^{n+1} \grk \,\uExt^n_{\rqgr R}(L,L') = \sum (-1)^{n+1} \grk \,\uExt^n_{\rqgr R}(L,L'[j])
\]
is obvious since  shifting  does not affect the rank.  Thus it suffices to prove that 
\begin{equation}\label{prop:0.65}
(L\dotms  L'[j]) = \sum (-1)^{n+1} \grk \,\uExt^n_{\rqgr R}(L,L'[j]).
\end{equation}

So, consider $N = \uExt^m_{\rqgr R}(L, L'[j]) =\uExt^m_{\rqgr R}(L, L')[j]$, for some $m\geq 0$.
Then after shifting by $[j]$, Lemma~\ref{lem0.3} shows that $N$ fits into an exact sequence 
$0 \to K \to N[-1] \overset{\timesg}{\to} N \to K' \to 0$ where $K$ and $K'$   can be zero or $ \kk$, depending on the choice of $m$.  
 Thus the kernel of  $\timesg$   on $N$  is contained in $N_{-1}$ and so $N_{\geq 0}$ is $g$-torsionfree.  
 Similarly, since $K'$ is concentrated in degree $0$,   $\timesg$ gives 
an isomorphism $N_n \overset{\sim}{\too} N_{n+1}$ for all $n \geq 0$.  Since $N$ is locally finite, 
it follows that $N_{\geq 0} \cong \kk[g]^{\oplus r}$ for some $r \geq 0$.  
In particular, $\dim_\kk N_0 = r = \grk N $ and so \eqref{prop:0.65} follows.  

If $p$ and $p'$ lie on distinct orbits then the same argument works, since now Lemma~\ref{lem0.3} implies that $N= \uExt^m_{\rqgr R}(L, L')
\cong N[-1]$ for each $m$. 
\end{proof}

The projective dimension of an $R$-module $L$ will be written $\pd_R(L)$.\label{pd-defn} 
We make the following easy observation.

\begin{lemma}\label{lem:pdL}
Let $R$ be an elliptic algebra and let $L$ be a line module with $\pd_{R^\circ}(L^\circ) < \infty$.
Then $\pd_{R^\circ}(L^\circ)=1$.
\end{lemma}

\begin{proof}
 By Lemma~\ref{lem:Rdual}(1),  $\uExt^n_{R}(L,R) =0$ for $n\not=1$ and so
 Lemma~\ref{lem0.1}  implies that     $\Ext^n_{  R^\circ}(L^{\circ}, R^\circ)=0$ for $n\not=1$.
 If $m=\pd_{R^\circ}(L^\circ) < \infty$ then it is easy  to see 
that $\Ext^m_{  R^\circ}(L^\circ,\,R^\circ)\not=0$, and it follows that $m=1$. 
\end{proof}

\begin{corollary}\label{lem0.7} Let $R$ be an elliptic algebra with line modules 
 $L$ and $L' $. Assume that  either $L^\circ$ or $(L')^{\circ}$ has finite projective dimension. Then
 $$\begin{array}{rl} 
 (L\dotms L') \  &=    \grk \uExt^1_{\rqgr R}(L,\,L')  - \grk \uHom_{\rqgr R}(L,\,L')\\
\noalign{\vskip 6pt} 
 &=\   \grk \uExt^1_{R}(L,\,L')  - \grk \uHom_{R}(L,\,L') \\
\noalign{\vskip 6pt}
 &=\  \dim_{\kk}\Ext^1_{R^\circ}(L^\circ,\,(L')^\circ)   -  \dim_{\kk}\Hom^1_{R^\circ}(L^\circ,\,(L')^\circ) .
 \end{array}$$
 \end{corollary}

 \begin{proof}  
 If $L^\circ$ has finite projective dimension, by Lemma~\ref{lem:pdL} we have $\pd_{R^\circ}(L^\circ)=1$.
 Thus, by Lemma~\ref{lem0.1}, for $i \geq 2$ we 
have $0 = \Ext^i_{  R^{\circ}}(L^{\circ}, (L')^{\circ}) = \grk \uExt^i_{\rqgr R}(L,L') = \grk \uExt^i_R(L,L')$ 
and the result follows from 
Proposition~\ref{prop:0.6} and Lemma~\ref{lem0.1}.   

If instead $(L')^{\circ}$ has finite projective dimension, then again  $(L')^{\circ}$ has 
projective dimension $1$, which forces its line ideal  $(J')^{\circ}$ to be projective. Lemmas~\ref{lem:Rdual}(1) and~\ref{lem0.1} 
still imply that
$\Ext^i_{R^{\circ}}(L^{\circ}, R^{\circ}) = 0$ for $i \geq 2$. Thus, as $(J')^{\circ}$ is a direct summand of a free module,  
$\Ext^i_{R^{\circ}}(L^{\circ}, (J')^{\circ}) = 0$ for $i \geq 2$ as well.
Then $\Ext^i_{  R^{\circ}}(L^{\circ}, (L')^{\circ}) \cong \Ext^{i+1}_{\  R^{\circ}}(L^{\circ}, (J')^{\circ}) = 0$ for $i \geq 2$, and again 
the result follows from Lemma~\ref{lem0.1} and Proposition~\ref{prop:0.6}.
\end{proof}

Unfortunately, when a (localised) line module $L^\circ$ has infinite projective dimension, the corollary can  fail and, 
indeed, $(L\dotms L)$ can even  be undefined; see  Corollary~\ref{eg-cor2} for an example of this phenomenon.  
However,  the higher Ext groups are not really relevant to our applications  of self-intersection and so we  can define  a modified intersection number  using just the first two terms in the alternating sum.

\begin{definition}\label{our-intersection}
Let $L$ and $L'$ be line modules over an  elliptic algebra $R$.  Then   define 
\[ (L \bigdot L') = - \grk \uHom_R(L, L') + \grk \uExt^1_R(L, L').\]
\end{definition}

Corollary~\ref{lem0.7} shows that  $(L\bigdot L') =  (L \dotms  L')$  provided one of  $(L)^{\circ}$ or $(L')^{\circ}$ has finite projective dimension.  
In general, the geometric interpretation of $(L\bigdot L')$ is more obscure, but as we see in Section~\ref{INTLINE}, our definition still correlates nicely 
with several other useful properties of the lines $L$ and $L'$.

In fact, the main examples we consider in the last part of the paper are elliptic algebras $R$ whose corresponding noncommutative projective schemes
 are \emph{smooth} in the sense that $\rqgr R$ has finite homological dimension.  
A geometric way of thinking about $R^\circ$ is to note that  the noncommutative 
scheme $\rqgr R$ has a closed subscheme
$\rqgr B \simeq \coh E$ which is a smooth elliptic curve. The category $\rmod R^{\circ}$ then represents the open complement of $E$.  Thus the following result is natural.

\begin{lemma}\label{lem:Dansmooth}
Let $R$ be an elliptic algebra.   
Then $\rqgr R$ is smooth if and only if $R^{\circ}$ has finite global dimension.
\end{lemma}

\begin{proof} If $\rqgr R$ is smooth then $R^{\circ}$ has finite global dimension by Lemma~\ref{lem0.1}, so suppose that $\rqgr R$ has infinite homological dimension. 
Thus, for any  $t\geq 3$  there exist $M,N\in \rgr R$ 
such that $\uExt^s_{\rqgr R}(M,N)\not=0$ for some $s\geq t$. By taking the beginning of a free resolution 
$0\to K\to F \to M\to 0$ and replacing $M$ by $K$ we can assume, possibly after increasing $s$,  that $M$ is
$g$-torsionfree. 
Similarly we may assume that $N$ is $g$-torsionfree.

Set $B=R/gR$. Then   $\rqgr B \cong \coh E$ has homological dimension $1$. Thus 
  Lemma~\ref{lem:extRextB} implies that 
$\uExt^i_{\rqgr R}(M,N/Ng) \cong \uExt^i_{\rqgr B}(M/Mg, N/gN) = 0$  for  $i  \geq 2$. Using
cohomology arising from the   exact sequence 
$0 \too N[-1] \overset{\timesg}{\too} N \too N/Ng \too 0$  it follows  that  $\timesg$ is injective on 
 $\uExt^i_{\rqgr R}(M,N)$ for all $i\geq 3$. In other words, $\uExt^s_{\rqgr R}(M,N)$ is $g$-torsionfree
 (and non-zero). 
 Finally,  by Lemma~\ref{lem0.1}  this implies that $\Ext^s_{R^\circ}(M^\circ,N^\circ)\not=0$. Since $t$ was arbitrary,  it follows that $R^\circ$ has infinite global dimension.
 \end{proof}

\section{Intersections of lines}\label{INTLINE}

 Fix an   elliptic algebra $R$ with $R/gR = B(E, \mc{M}, \tau)$, and 
let $L = R/J$ and $L' = R/J'$ be two right 
$R$- line modules, possibly isomorphic, with $\Div L = p$ and $\Div L' = p'$.
 In this section, we study alternative characterisations of the intersection number 
 $(L \bigdot L')$, as defined in Definition~\ref{our-intersection}.  
 In particular, we show that $(L\bigdot L)=-1$ 
 if and only if $\uExt_R^1(L,L)=0$, and give similar conditions for when $(L\bigdot L')=0.$ 

We begin with a number of useful observations.
First, consider the exact sequences
 \[ 0 \to \uHom_R(L,L') \to \uHom_R(R, L') \to \uHom_R(J, L') \to \uExt^1_R(L,L') \to 0,\]
 and
 \[ 0 \to \uHom_R(J,J') \to \uHom_R(J, R) \to \uHom_R(J, L') \to \uExt^1_R(J, J') \to 0.\ \]
(The second  sequence is exact because $J$ is CM, by Lemma~\ref{lem:Rdual}.)  
We know that $\hilb \uHom_R(J, R) = \hilb R + s/(1-s)^2$ by 
the left-sided version of 
Lemma~\ref{lem:Rdual}(3),   
and of 
course $\uHom_R(R, L') = L'$ has Hilbert series $1/(1-s)^2$.  Thus we obtain the useful equation:
\begin{multline} \label{eq:main}
 \hilb \uExt^1_R(L,L')-\hilb \uHom_R(L, L') = \hilb \uHom_R(J,L') - \frac{1}{(1-s)^2} \\
= \hilb \uExt^1_R(J,J') + h_R - \frac{1}{(1-s)} - \hilb \uHom_R(J,J').
\end{multline}
The power series on the left of \eqref{eq:main} will recur often, so we define:
\begin{equation}\label{XLL-defn}
 X(L,L') = \hilb \uExt^1_R(L,L')-\hilb \uHom_R(L, L').
 \end{equation} 

Next,   by applying $\uHom_R(L, -)$ to the short exact sequence $0 \to L'[-1] \overset{\timesg}{\to} L' \to M_{p'} \to 0$, 
and using Lemma~\ref{lem:extRextB}  one obtains  the  exact  sequence:
\begin{multline}\label{eq:extLL} 0 \too \uHom_R(L,L')[-1] \overset{\timesg}{\too} \uHom_R(L, L') \too   
\uHom_B(M_p, M_{p'})  \\ \too \uExt^1_R(L, L')[-1] \overset{\timesg}{\too} \uExt^1_R(L, L') \overset{\delta}{\too} \uExt^1_B(M_p, M_{p'}). 
\end{multline}
This gives  the Hilbert series equation
\begin{equation}
\label{eq:EH}
X(L, L')   = \frac{C - H}{(1-s)}  \qquad \text{ for}\ H = \hilb \uHom_B(M_p, M_{p'}) \  \ \text{and} \  \ 
C = \hilb \im \delta.
\end{equation}
The possibilities for $H, C$  and $E = \hilb \uExt^1_B(M_p, M_{p'})$\label{HEC-defn}
  are quite limited:  by Proposition~\ref{prop:Hompt}, if  
$p = \tau^j(p')$ for some $j \geq 0$, then $H = s^j$ and $E = s^{-1} + s^j$; otherwise $H = 0$ and $E = s^{-1}$.  In 
any case $0 \leq C \leq E$.

Note that  \eqref{eq:extLL} implies that $\uHom_R(L,L')$ is $g$-torsionfree. 
Moreover,  by Proposition~\ref{prop:Hompt}(1) the  map $\timesg$ on $\uExt^1_R(L,L')$  
has a finite-dimensional kernel.  By
Lemma~\ref{lem:barhom1}, 
this implies that the $g$-torsion submodule of $\uExt^1_R(L,L')$ is  also finite-dimensional.  
Since  $\dim_\kk  \im \delta<\infty$  
and $\uExt^1_R(L,L')$ is left bounded, \eqref{eq:extLL} implies that $\dim_{\kk} \uExt^1_R(L,L')_n$ is constant 
for $n \gg 0$.  
Thus $\grk \uExt^1_R(L,L')=\dim_\kk \uExt^1_R(L,L')_n$ for all $n \gg 0$.
Obviously the analogous result holds 
for $\uHom_R(L, L')$, and combined  with \eqref{eq:EH} this implies that 
\begin{align}
\label{eq:int-hs}
(L \bigdot L')  & = \grk \uExt^1_R(L,L') - \grk \uHom_R(L,L')  \\ 
\nonumber & = \dim_\kk \uExt^1_R(L,L')_n - \dim_\kk \uHom_R(L,L')_n, \ \text{ for } n \gg 0 \\
\nonumber & = \text{the sum of the coefficients of}\ C - H.
\end{align}
We may think of  $X(L,L') $ as a refined version of $(L \bigdot L')$, since $(L \bigdot L') = (1-s) X(L,L')|_{s=1}$ by \eqref{eq:EH} and \eqref{eq:int-hs}.  

As a consequence of these calculations, we can already see that the intersection number of two lines on an elliptic algebra lies in a 
quite limited range.  If $p = \tau^j(p')$ for some $j \geq 0$, then $H = s^j$ and $C \leq s^{-1} + s^j$, while otherwise 
$H = 0$ and $C \leq s^{-1}$.  From \eqref{eq:int-hs} we conclude that 
\begin{equation}\label{intersection-range}
(L \bigdot L') \in \{-1,0,1\},
\end{equation}
 where the value $-1$ can only occur  if $p = \tau^j(p')$ for some $j \geq 0$.  In Lemma~\ref{lem:new}  we will refine this observation to 
 show that in fact $(L \bigdot L') = -1$ forces   $L $ to be isomorphic to $L'$.

Next, since $J$ and $J'$ are $g$-divisible by Lemma~\ref{lem:Rdual}(2),  we may use Lemma~\ref{lem:extRextB} 
to get  the exact sequence
\begin{equation}
\label{eq:JJ}
0 \too \uHom_R(J, J')[-1] \overset{\timesg}{\too} \uHom_R(J, J') \too \uHom_B(\bbar{J}, \bbar{J'}) \overset{\alpha}{\too} 
 \uExt^1_R(J, J')[-1] \overset{\timesg}{\too} \uExt^1_R(J, J') \too \cdots .
\end{equation}
The calculation of the Hilbert series of $\uHom_B(\bbar{J}, \bbar{J'})$ is straightforward.  Since $\bbar{J} = \bigoplus_{n \geq 0} H^0(E, \sM_n(-p))$ 
and  similarly for $J'$,  
by Lemma~\ref{lem:uhomB} we get 
$\uHom_B(\bbar{J}, \bbar{J'}) = \bigoplus_{n\geq 0} H^0(E, \sM_n(-p' + \tau^{-n}(p))). $
Thus 
\beq\label{epsilon}
\hilb \uHom_B(\bbar{J}, \bbar{J'}) = \hilb B - 1 + \epsilon_{p,p'} \qquad \text{where}\ \epsilon_{p,p'} = \begin{cases} 1 & p = p' \\ 0 & p \neq p' \end{cases}. 
\eeq 
From \eqref{eq:JJ} and Lemma~\ref{lem:ideal}(1) 
we have 
$\bbar{\uHom_R(J, J')} \subseteq \uHom_B(\bbar{J}, \bbar{J'})$ with cokernel isomorphic to $\im \alpha$.  Thus 
we conclude that 
\begin{equation}
\label{eq:homJJ}
\hilb \uHom_R(J, J')   \ = \ \frac{\hilb B - 1 + \epsilon_{p, p'} - \hilb \im \alpha}{(1-s) }
\  =\  \hilb R  +  \frac{(\epsilon_{p, p'}-1)}{(1-s)} - \frac{(\hilb \im \alpha)}{(1-s)}.
\end{equation}

We next want to characterise when the  various intersection numbers   occur in 
\eqref{intersection-range}.
We begin with  $(L\bigdot L)$ for a single line ideal $L$  and  are mostly interested in when 
$(L \bigdot L)=-1$.

\begin{theorem}\label{thm:one}
Let $L = R/J$ be a  line module of an elliptic algebra $R$.  
Consider the following conditions:
\begin{enumerate}

\item $\uExt^1_R(J,J) = 0$.
\item $\uExt^1_R(L,L) = 0$.
\item $\hilb \uEnd_R(J) =\hilb R$.
\item$\hilb \uHom_R(J, L)  =  s(1-s)^{-2}$.
\item $(L \bigdot L) = -1$.
\end{enumerate}
Then:
\begin{itemize}
\item[(a)] $(1)$ $\iff$ $(2)$ $\iff$ $(4)$ $\iff$ $(5)$ $\Rightarrow$ $(3)$.  
\item[(b)]  If $J^\circ$ is projective then  all five are equivalent.  
\end{itemize}
\end{theorem}

\begin{proof}  (a)
By Lemma~\ref{lem:red0}(2),   $\hilb \uHom_R(L,L) = (1-s)^{-1}$.   
Thus adding $ (1-s)^{-1}$ to \eqref{eq:main} we obtain
\beq
\label{eq:main2}
 \hilb \uExt^1_R(L,L)  = \hilb \uHom_R(J,L) - \frac{s}{(1-s)^2}= \hilb \uExt^1_R(J,J) + [h_R  - \hilb \uHom_R(J,J)].
\eeq
This shows immediately that $(2)$ $\iff$ $(4)$.  The final term $[h_R  - \hilb \uHom_R(J,J)]$ of \eqref{eq:main2}
  is non-negative by \eqref{eq:homJJ}.  
Since $\hilb \uExt^1_R(J,J)$ is obviously non-negative, it follows  that $(2)$ implies both 
$(1)$ and $(3)$.

Now if $(1)$ holds, then $\alpha = 0$ in \eqref{eq:JJ} and so \eqref{eq:homJJ} immediately implies $(3)$.  But \eqref{eq:main2} again shows that $(1)$ 
and $(3)$ together imply $(2)$.

If $(2)$ holds, then since $\uHom_R(L,L) = \kk[g]$  by Lemma~\ref{lem:red0}, we certainly have $$(L \bigdot L) = \grk \uExt^1_R(L,L) - \grk \uHom_R(L,L) = -1$$
and  $(5)$  holds.  Conversely, if $(5)$ holds then using \eqref{eq:int-hs} we  see that 
$H = 1$ and $C = 0$ in \eqref{eq:EH}.  Thus the map $\uExt^1_R(L, L)[-1] \overset{\timesg}{\too} \uExt^1_R(L, L)$ is surjective, and since 
$\uExt^1_R(L,L')$ is left bounded it must therefore be  zero.   
Thus $(2)$ holds, which completes the proof of (a).  

 (b) Suppose that  $J^\circ$ is projective and that $(3)$ holds.  From \eqref{eq:homJJ} we conclude that $\alpha =0$ in \eqref{eq:JJ} and so  $\uExt^1_R(J, J) $ is $g$-torsionfree.
But   $\uExt^1_{R^\circ}(J^\circ,J^{\circ}) = 0$ and so  $\uExt^1_R(J, J)$ is   $g$-torsion by Lemma~\ref{lem0.1}.
Thus $\uExt^1_R(J,J) = 0$ and  $(1)$ holds.  
\end{proof}

\begin{remark}\label{smooth-thm:one} (1) Suppose that  $R^\circ$ has finite global dimension, or, equivalently by Lemma~\ref{lem:Dansmooth}, 
that $\rqgr R$ is smooth.  
Then $J^\circ$ is automatically projective by Lemma~\ref{lem:pdL} and so part (b) of the theorem applies.  

(2) We  know of no example of a line module $L$ of an elliptic algebra $R$, with $\rqgr R$ smooth, for which  $(L\bigdot L) \not= -1$
 and we actually conjecture that none exist.   
 For a (non-smooth) example with   $(L\bigdot L)\not=-1$, see Corollary~\ref{eg-cor1}.
\end{remark}

In the rest of the section, we  consider the intersection number  of two distinct line modules. Although these results will not be directly 
relevant for this paper, they will be useful in \cite{RSS4}, and as the proofs are similar to that of Theorem~\ref{thm:one}, it is appropriate to  include them here.  

The following easy lemma will be used a number of times.
 
\begin{lemma}
\label{lem:homJ}
Let $R$ be  an  elliptic  algebra  with two line modules  $L = R/J$ and $L' = R/J'$. 
 Then 
$\uHom_R(J,J')_0 \subseteq R_0 = \kk$.
Moreover $\uHom_R(J,J')_0  = \kk$  $\iff$  $J = J'$ $\iff$  $L \cong L'$.
\end{lemma}
\begin{proof}
Certainly $\uHom_R(J,J')\subseteq \uHom_R(J, R) = J^*$.  By the left-handed   
analogue of Lemma~\ref{lem:Rdual}(3), $(J^*)_0 = R_0 = \kk$.
Furthermore, if $J = J'$ then $\kk \subseteq \uHom_R(J, J')_0$ and so $\uHom_R(J,J')_0 = \kk$.  Conversely, 
if $\kk \subseteq \uHom_R(J,J')$ then $J \subseteq J'$ and so $J = J'$ since $J$ and $J'$ have the same Hilbert series.
The last equivalence is immediate since a line module is determined up to isomorphism by its line ideal.
\end{proof}

We next refine \eqref{intersection-range}  by showing that there 
are only two possible values of   $(L \bigdot L')$ for 
 non isomorphic lines.

\begin{lemma}\label{lem:new}
Let $L\not\cong L'$ be line modules over an elliptic algebra $R$.
Then
 $X=X(L,L')$ is equal to  either 
$0$, $s^{-1}(1-s)^{-1}$,  or $s^{-1} + \dots +s^{j-1}$ for some $j \geq 0$.   In particular, $(L \bigdot L')\in \{0,1\}$.
\end{lemma}

\begin{proof}
Adopt the notation of \eqref{eq:EH}.  
If 
$H=0$ then \eqref{eq:EH} and Proposition~\ref{prop:Hompt} imply that $C \leq s^{-1}$ and $X$ is either 0 or $s^{-1}(1-s)^{-1}$.

So suppose that $H \neq 0$. Then by Proposition~\ref{prop:Hompt}(1),  $H = s^j$ for some $j > 0$ and this happens precisely when $p=\Div L =\tau^j (p')$ for $p'=\Div L'$.   
By Proposition~\ref{prop:Hompt}(2) and \eqref{eq:extLL},  $C\leq E = s^{-1} + s^j$.
Let $J, J'$ be, respectively, the line ideals of $L$ and $L'$ and note that $J \neq J'$ since $L \not\cong L'$.
Thus $\uHom_R(J,J')_0 = 0$ by Lemma~\ref{lem:homJ}.  It follows 
from \eqref{eq:homJJ} that $\hilb \uHom_R(J,J') \leq h_R - (1-s)^{-1}$.
From \eqref{eq:main}, $X \geq \hilb \uExt^1_R(J,J') \geq 0$.
This forces $ C \neq 0$, and so $C$ is one of $s^j$, $s^{-1}$ or $s^{-1} + s^j$.
Now use \eqref{eq:EH} again to get the desired possibilities for $X$.  

The possibilities for $(L \bigdot L')$ are an immediate consequence.  
\end{proof}

\begin{lemma}\label{lem:Extfinite}
Let $R$ be an elliptic algebra.  Let $J, K$ be finitely generated $g$-divisible reflexive right $R$-submodules 
of $R_{(g)}$. Assume  that either 
\begin{itemize}\item[(a)] $J^\circ$ is projective; or 
\item[(b)] $J$ is CM and $K^\circ$ is projective.  \end{itemize}
Then $\uExt^1_R(J,K)$ is finite-dimensional.
\end{lemma}

\begin{proof} 
Consider    $H = \uExt^1_R(J,K)$.  By Lemma~\ref{lem:barhom}(2), the $g$-torsion subspace of $H$ 
is finite dimensional.     
It therefore  suffices to prove that $H$ is $g$-torsion.  By Lemma~\ref{lem0.1}, it is then enough to show 
that  $\Ext^1_{R^\circ} (J^\circ, K^\circ)=0$.  This is trivial if $J^\circ$ is projective.  If  
$K^\circ$ is projective, then $\Ext^1_{R^\circ}(J^\circ, K^\circ)$ is a direct summand of a sum of 
copies of  $\Ext^1_{R^\circ}(J^\circ, R^\circ)$.  
As $J$ is CM,    Lemma~\ref{lem0.1} again implies that $\Ext^1_{R^\circ}(J^\circ, R^\circ)=0$
and hence   $\Ext^1_{R^\circ} (J^\circ, K^\circ)=0$.
\end{proof}

We next  characterise distinct  lines with $(L \bigdot L') =1$.

\begin{theorem}\label{thm:three}
Let $L = R/J$, $L' = R/J'$ be   line modules over an elliptic algebra $R$, with $L \not\cong L'$.  
Consider the following conditions:
\begin{enumerate}
\item $\hilb \uExt^1_R(J,J') = s^{-1} + 1$.
\item $X(L,L') =  s^{-1}(1-s)^{-1}$.
\item $\hilb R - \hilb \uHom_R(J,J') = (1+s)(1-s)^{-1}$. 
\item $\hilb \uHom_R(J, L')  =  s^{-1}(1-s)^{-2}$.
\item $(L \bigdot L') = 1$.
\end{enumerate}
Then:
\begin{itemize}
\item[(a)] $(1)$ $\Rightarrow$ $(2)$ $\iff$ $(4)$ $\iff$ $(5)$, while   $(1)$ $\Rightarrow$ $(3)$.  
\item[(b)]  If either $J^\circ$ is projective with $(L \bigdot L) = -1$, or $(J')^\circ$ is projective
with  $(L'\bigdot L') = -1$, then  all five conditions are equivalent.  
\end{itemize}
\end{theorem}
\begin{proof}
(a) 
Once again \eqref{eq:main}  implies that $(2)$ $\iff$ $(4)$,  while 
Lemma~\ref{lem:new} gives $(2)$ $\iff$ 
$(5)$.  

 $(1)$ $\Rightarrow$ $(3)$. Suppose that $(1)$ holds and consider the part of the exact sequence 
\eqref{eq:JJ} given by 
\[  \uHom_R(J,J') \too 
\uHom_B(\overline{J}, \overline{J'}) \overset{\alpha}{\too} \uExt^1_R(J,J')[-1] \overset{\timesg}{\too} \uExt^1_R(J,J').
\]
By assumption $\uExt^1_R(J,J')=\kk+\kk[1]$ is $g$-torsion, and so  the kernel of   $\timesg$ 
 on $\uExt^1_R(J,J')$ contains at least the highest degree piece  $\uExt^1_R(J,J')_0$.
Thus $\hilb\im \alpha$ equals either $s$ or $1+s$.

Let $p = \Div L$ and $p' = \Div L'$.  If $p\neq p'  $, then $\uHom_B(\overline{J}, \overline{J'})_0 = 0$ 
and so $\hilb \im \alpha = s$.  If $p = p'$ then $\uHom_B(\overline{J}, \overline{J'})_0 = \kk$, whereas 
$\uHom_R(J, J')_0 = 0$ by Lemma~\ref{lem:homJ}.  Thus $\hilb \im \alpha = 1+s$.  
In either case,  $\hilb \im \alpha = \epsilon_{p,p'}+s$, in the notation of  \eqref{epsilon}.  
Thus  $(3)$ follows from \eqref{eq:homJJ}.

 $(1)$ $\Rightarrow$ $(2)$.  If $(1)$ and therefore $(3)$ hold, then  $(2)$ follows from \eqref{eq:main}.

\medskip
(b)  $(3)$ $\Rightarrow$ $(1)$.   Suppose  that  $(3)$ holds.
Recall that $J$ and $J'$ are CM, $g$-divisible, and reflexive, by Lemma~\ref{lem:Rdual}(2).   Thus the assumption in (b) ensures that
 the hypothesis of Lemma~\ref{lem:Extfinite} holds.  Then  $\dim_\kk \uExt^1_R(J,J')<\infty$ by that lemma,  and so \eqref{eq:main} implies that 
$\dim_\kk \uExt^1_R(L, L')_n  - \dim_\kk \uHom_R(L, L')_n = 1$ for $n \gg 0$. Hence $(L \bigdot  L') = 1$ 
by \eqref{eq:int-hs}.  Thus 
$(3)$ implies $(5)$ and hence  $(2)$.  But $(2)$ and $(3)$ together force 
$\hilb \uExt^1_R(J,J') = s^{-1} + 1$ 
by \eqref{eq:main}, and so $(1)$ holds.   

 $(2)$ $\Rightarrow$ $(3)$.  Finally, assume that  $(2)$ holds. Then, as $\dim_\kk \uExt^1_R(J,J')<\infty$,   \eqref{eq:main} and  \eqref{eq:homJJ} together imply that 
$\dim_\kk \im \alpha = 1+\epsilon_{p,p'}$.  We claim that  this 
forces $\hilb \im \alpha = \epsilon_{p,p'} + s$.  
First note that $\dim (\im \alpha)_0 = \epsilon_{p,p'}$. So if $\hilb \im \alpha \neq \epsilon_{p,p'}+s$,
   the cokernel of $\uHom_R(J,J') \to \uHom_B(\overline{J}, \overline{J'})$ is one-dimensional in 
 some degree  $> 1$.   Thus 
  $\overline{\uHom_R(J,J')}\supseteq \uHom_B(\overline{J}, \overline{J'})_1 = H^0(E, \sM_n(-p' + \tau^{-1}(p)))$, and 
it is also an $\Big(\overline{\uEnd_R(J', J')},\,\, \overline{\uEnd_R(J,J)}\Big)$-bimodule.  By hypothesis, either
 $L$ or  $L'$ satisfies the equivalent conditions 
in Theorem~\ref{thm:one}; say it is $L$.  Then 
$\overline{\uEnd_R(J,J)} = \uEnd_B(\overline{J}, \overline{J})$
 is a full TCR, say   $B'=B(E, \mc{M}(-p), \tau)$.  By Lemma~\ref{lem:uhomB},
  $\uHom_B(\overline{J}, \overline{J'})_{\geq 1} $ is generated in degree $1$ as a right $B'$-module and 
  thus  $\overline{\uHom_R(J,J')} = \uHom_B(\overline{J}, \overline{J'})_{\geq 1}$, a contradiction.  Similarly,
   if $\overline{\uEnd_R(J', J')}$ is a full TCR we get a contradiction by viewing  $\uHom_B(\overline{J}, \overline{J'})_{\geq 1}$ as a left module.  
This proves the claim that  $\hilb \im \alpha = \epsilon_{p,p'}+s$ which,  by \eqref{eq:homJJ},  implies $(3)$.
\end{proof}

Finally, we characterise lines $L$ and $ L'$ with $(L \bigdot L') = 0$, although here we need to assume   that
 $\Div L \neq \Div L'$.

\begin{theorem}\label{thm:two}
Let $L = R/J$ and $L' = R/J'$ be   line modules over an elliptic algebra $R$ with divisors 
$\Div L = p\not= p'= \Div L'$.   Consider the following conditions.
\begin{enumerate}
\item $\uExt^1_R(J,J') = 0 $.
\item $\hilb \uExt^1_R(L,L')=\hilb  \uHom_R(L, L')$.
\item $\hilb R - \hilb \uHom_R(J,J') = (1-s)^{-1}$.  
\item $\hilb \uHom_R(J, L')  = (1-s)^{-2}$.
\item $(L \bigdot L') = 0$.
\end{enumerate}
Then:
\begin{itemize}
\item[(a)]$(1)$ $\iff$ $(2)$ $\iff$ $(4)$ $\iff$ $(5)$ $\Rightarrow$ $(3)$.  
\item[(b)]  If $J^\circ$ or $(J')^\circ$ is projective then  all five are equivalent.  
\end{itemize}
\end{theorem}

\begin{proof}
The proof of this result is very similar to the proof of Theorem~\ref{thm:one}.  Recall \eqref{eq:main}:
\[ X(L,L') \ = \  \hilb \uHom_R(J,L') - \frac{1}{(1-s)^2} 
\ = \  \hilb \uExt^1_R(J,J') +\mathscr{H} , \]  
where $\mathscr{H} = h_R -  (1-s)^{-1} - \hilb \uHom_R(J,J').$
In this case, since $p \neq p'$, \eqref{eq:homJJ} shows that $\mathscr{H}\geq 0$.  Thus just as in 
the proof of Theorem~\ref{thm:one},  $(2)$ $\iff$ $(4)$ is immediate, and $(2)$ implies $(1)$ and $(3)$.
Again, by \eqref{eq:homJJ}, $(3)$ is equivalent to $\alpha = 0$ in \eqref{eq:JJ}, and so $(1)$ implies $(3)$; 
thus $(1)$ implies $(2)$   by \eqref{eq:main}.

If $(2)$ holds, so $X(L,L') = 0$, then certainly $(L \bigdot L') = 0$, by \eqref{eq:int-hs}, and so  $(5)$ holds.

Finally, if $(5)$ holds, then \eqref{eq:int-hs} implies that  
$\dim_\kk \uExt^1_R(L,L')_n - \dim_\kk \uHom_R(L, L')_n = 0$ for 
all $n \gg 0$. This  shows that,  in \eqref{eq:main}, the non-negative Hilbert series
   $\mathscr{H}$  and $\hilb \uExt^1_R(J,J')$
   must be polynomials in $s$.  But \eqref{eq:homJJ} shows that $\mathscr{H}$ is either $0$ or a  
    multiple of $1/(1-s)$.  Thus  $\mathscr{H}=0$, and so $(3)$ holds.  
    As we already saw, $(3)$ implies that $\alpha = 0$ in \eqref{eq:JJ}, which in turn implies that $Z=\uExt^1_R(J,J')$ is $g$-torsionfree. 
    But as we have already shown that  $\dim_\kk Z<\infty$, 
     this forces $Z = 0$.  Thus $(5)$ implies $(1)$.
This finishes the proof of (a).

  (b) In this case,   Lemma~\ref{lem:Rdual}(2) and Lemma~\ref{lem:Extfinite}  
imply that  $Z=\uExt^1_R(J, J')$ is finite-dimensional. 
    Since we just saw  that $(3)$ implies that $Z$ is $g$-torsionfree,  $(3) \Rightarrow (1)$ holds  in this case.
\end{proof}

 
\section{Blowing down elliptic algebras and their modules}\label{BLOWDOWN}

As mentioned in the introduction  (see Proposition~\ref{prop1} and Theorem~\ref{CCC})
two fundamental and inverse constructions in birational geometry  are the concepts of  blowing up a closed point $p$ on  
 a smooth projective surface  and, conversely,    blowing   down lines of self-intersection $-1$. In the noncommutative universe
  one again has a notion of blowing up, coming from \cite{VdB-blowups} and \cite{R-Sklyanin} and at least if the set-up is
   ``smooth enough'' then one obtains an exceptional line module $L$ with $(L\bigdot L)=-1$. In this section we prove 
one of the main results of the paper by giving a noncommutative analogue of Castelnuovo's Theorem: if $R$ is an elliptic algebra
 with a line module $L$ satisfying $(L\bigdot L)=-1$, then one can blow $L$ down to get a second elliptic algebra. Moreover, this 
 operation is the inverse of blowing up. In fact we can do a little better and also obtain a result that works without smoothness assumptions, 
 by  replacing the requirement that $(L\bigdot L)=-1$ by the assumption that the corresponding line ideal $J$ satisfies 
$\hilb \uEnd_R(J)=\hilb R$. The details are given in Theorems~\ref{thm:I.1} and \ref{thm:inverse}.

We will need the following technical result about images of direct sums of a single line module.

\begin{proposition}
\label{prop:image}
Let $L$ be a right line module over an elliptic algebra   $R$.  Let $\II$ be an index set, 
and for each $i \in \II$ 
choose an integer $a_i \in \mb{Z}$.  
 Assume $\{a_i : i\in\II\}$ is bounded below.  
Let $A\subseteq \LL = \bigoplus_{i \in \II} L[-a_i]$
be  any submodule  such that $N = \LL/A$ is GK-2 pure.  Then there is $\JJ \subseteq \II$ 
such that $\LL$ is an internal  direct sum  $\LL = A \oplus \bigoplus_{i \in \JJ} L[-a_i]$.  In particular,
 $N \cong \bigoplus_{i \in \JJ} L[-a_i]$.  
\end{proposition}

\begin{proof}
We claim first that there exists a subset $\JJ \subseteq \II$ which is 
minimal under inclusion in the family 
$\mc{S}=\left\{ \KK\subseteq \II : \LL = A + \bigoplus_{i \in \KK} L[-a_i]\right\}$.  
   Indeed, since $N$ is left bounded, the graded 
Nakayama lemma (which does not require finite generation) 
applies and shows that $\KK \in \mc{S}$ if and only if the induced map 
$\psi: \bigoplus_{i \in \KK} L[-a_i]/L[-a_i]R_{\geq 1} \to N/NR_{\geq 1}$ is surjective.  
 Since   $L$  is cyclic, $\dim_\kk L[-a_i]/L[-a_i]R_{\geq 1}=1$.  Thus any   $\JJ$ for which $\psi$ is an
  isomorphism of $\kk$-spaces will be   minimal  
in $\mc{S}$, proving the claim.

Fix some such   $\JJ$.  We will prove that $\LL = A \oplus \bigoplus_{i \in \JJ} L[-a_i]$.  Suppose that  this fails;
thus  $0 \neq A\, \cap\, \bigoplus_{i \in \JJ} L[-a_i]$.   
 Write $L[-a_i]=\alpha_i R$, for homogeneous $\alpha_i$ and 
choose   $0 \neq u  = (\alpha_{i} x_{i})  \in A\, \cap\, \bigoplus_{i \in \JJ} L[-a_i]$,    
  for some homogeneous elements $x_{i} \in R$.    Let $\II^1 $ be the (finite)
 set $ \{ i \in \JJ :   \alpha_{i} x_{i} \neq 0\}$.  Without loss of generality, we may choose such 
 $u$ so that $|\II^1|$ is minimal.  Let  $\LL^1 = \bigoplus_{i\in \II^1} L[-a_i]$.

 Let $J = \Ann_R(u)$ and, for $i \in \II^1$, set $J_i = \Ann_R(\alpha_i x_i)$.   If $r \in J_i$,  then 
 $ur$ has strictly more zero entries than $u$, and so  the choice of $u$  forces $ur=0$.  So 
 $J_i \subseteq J  \subseteq J_i$.  Thus $J_i = J$ is independent of the choice of $i \in \II^1$.  
Fixing  an arbitrary $i \in \II^1$, we have $uR \cong R/J = R/J_i \cong \alpha_i x_i R \subseteq L[-a_i]$.  
 By \cite[Lemma~8.9(2)]{R-Sklyanin}, there is a (shifted) line module $L' \subseteq L[-a_i]$, containing
  $\alpha_i x_i R$, for which  $L'/(\alpha_i x_i R) = F$ is finite-dimensional. 
Since $uR \cong \alpha_i x_i R$, there is an injective homomorphism $uR \hra L'$ with  a finite-dimensional cokernel.
 Since $\uExt^1_R(F, \LL^1) = 0$ by Lemma~\ref{lem:Rdual}(5), the canonical injection  $uR \hra \LL^1$
  lifts to an injection $L' \hra \LL^1 \subseteq \LL$.  In other words, there is some $u' = (\alpha_i x'_i)$ 
  so that $uR \subseteq u'R \cong L'$, where $u'R/uR$ is finite-dimensional.  As $\LL^1$ is torsionfree,  
  $\alpha_i x'_i \neq 0$ if and only if $i \in \II^1$.  The  argument from the beginning 
  of the paragraph 
shows that $\Ann_R(\alpha_\ell x'_\ell) = \Ann_R(u')$, and hence that  $\alpha_\ell x'_\ell R \cong L'$, for all $\ell \in \II^1$.

Choose   $j\in \II^1$ such that $a_j = \max\{ a_i : i \in  \II^1\}$ and write $L'=\beta R$ for some homogeneous $\beta$.  
For any other $i \in \II^1$, by Lemma~\ref{lem:red0}(3) we may identify 
$\kk \cdot g^{a_j-a_i}=\Hom_{\rgr R}(L[-a_j], L[-a_i]) $, where  $g^{a_j-a_i}$ maps $\alpha_j \mapsto \alpha_i g^{a_j-a_i}$.
Now   
consider the following two    maps in $\Hom_{\rgr R}(L',L[-a_i])$:
\[ f:  \beta x \ \mapsto\ \alpha_i x'_i x,  \quad \quad \text{and}
 \quad \quad f': \beta x \mapsto\ \alpha_j x'_j x \  \mapsto\ \alpha_i g^{a_j-a_i} x'_j  x . \]
 Since $\Hom_{\rgr R}(L', L[-a_i]) \neq 0$, by Lemma~\ref{lem:red0}(2)  it is one-dimensional.
   Thus there is $\lambda_i \in \kk^*$ so that $f = \lambda_i f'$; that is
 $\alpha_i x'_{i} = \lambda_i \alpha_i g^{a_j-a_i} x'_j$.  Thus $u' = (\lambda_i \alpha_i g^{a_j-a_i} )x'_j$, 
whence  $L'\cong u' R\subseteq K = v R$, where $v = (\lambda_i \alpha_i g^{a_j-a_i})$. Clearly 
$K \cong L[-a_j]$.  
Since $L$ is $2$-critical, $\GKdim K/uR \leq 1$.  Since $uR \subseteq A$ and $\LL/A$ is $2$-pure,
 we must have $K \subseteq A$.  
Now it is easy to see that $\LL^1 = \bigoplus_{i \in \II^1} L[-a_i] = 
K \oplus \bigoplus_{i \in \II^1 \setminus \{ j \}} L[-a_i]$.  This implies 
that $A + \bigoplus_{i \in \JJ \setminus \{ j \}} L[-a_i] = \LL$, contradicting the choice of $\JJ$.  
Hence 
$\LL = A \oplus \bigoplus_{i \in \JJ} L[-a_i]$, as required.
\end{proof}

If $L$ is a line module over an elliptic algebra $R$, we now use it to define the largest extension 
$  \wt{K}$ of a   reflexive module $K\subset R_{(g)}$ by sums of shifts of $L$.\label{wt-defn}

\begin{lemma}
\label{lem:tilde}
Let $R$ be an elliptic algebra, and let $L = R/J$ be a right $R$-line module  with $\Div L = p \in E$.    
Let $K \subseteq R_{(g)}$ be a graded finitely generated $g$-divisible reflexive right $R$-module.  
Set  
\begin{equation}\label{eq:tilde}
\wt{K}  \ = \ \wt{K}_L\ = \ \sum_{\alpha}\Bigl\{  N_{\alpha} :   K \subseteq N_{\alpha}\subset Q_{\gr} (R) 
\ \text{ with }\   N_{\alpha}/K \cong L[-i_\alpha] \ \text{ for some } i_\alpha\in \ZZ\Bigr\}.
\end{equation}  Then the following hold.
\begin{enumerate}
\item $\wt{K} = \uHom_R(J, K) R$ and $\wt{K} \subseteq R_{(g)}$.

\item   As right $R$-modules,   $\wt{K}/K \cong \bigoplus_{i \in \mb{Z}} L[-i]^{\oplus a_i}$ 
for some $a_i \geq 0$. Hence $\hilb \wt{K}/K   = p(s)(1-s)^{-2}$  for 
 $p(s) =  \sum_{i\in\ZZ}a_is^i$. 
 \item Moreover,  $\hilb \uExt^1_R(L,K) = p(s)(1-s)^{-1}$.  
\item If $\uExt^1_R(L,L)  =0$ then $\uExt^1_R(L, \wt{K}) = 0$.
\end{enumerate}
\end{lemma}

\begin{proof}  Throughout the proof $N_\alpha$ will denote a module satisfying the properties defined by \eqref{eq:tilde}.

(1) If $N_\alpha   \subseteq Q_{gr}(R)$ satisfies $N_\alpha  /K \cong L[-i]$, then 
$N_{\alpha}  = xR + K$ for some $x \in Q_{gr}(R)_i$.  Then $xJ \subseteq K$ and so 
$N_\alpha   \subseteq \uHom_R(J, K) R$.   Thus $\wt{K} \subseteq \uHom_R(J, K) R$.
Conversely, 
if $x \in \uHom_R(J, K)_i\subset Q_{\gr}(R)$ then $(xR + K)/K$ is a homomorphic image of 
$(R/J)[-i] = L[-i]$.  As $K$ 
is reflexive, $(xR+K)/K$ is either $0$ or $2$-pure by Lemma~\ref{lem:reflexive}(1).  Therefore, because $L$ is 
$2$-critical, either $x \in K$ or else $(xR+K)/K \cong L[-i]$.  
In either case,  $x \in \wt{K}$ by the definition of $\wt{K}$.  Thus $\uHom_R(J, K) R \subseteq \wt{K}$ 
since $\wt{K}$ 
is a right $R$-module.    

Since $J$ and $K$ are  
$g$-divisible,  $\uHom_R(J,K) \subseteq R_{(g)}$ by Lemma~\ref{lem:ideal}(1), 
and so $\wt{K} \subseteq R_{(g)}$.  

(2) Clearly  $\wt{K}/K$ is a homomorphic image of 
$\bigoplus_{\alpha} N_{\alpha}/K \cong \bigoplus_{\alpha} L[-i_{\alpha}]$.  
Once again,   Lemma~\ref{lem:reflexive}(1) implies that $\wt{K}/K$ 
is either $2$-pure or $0$.    
Since $J$ and $K$ are left bounded and locally finite, so is $\uHom_R(J,K)R =\wt{K}$.  Thus there is a 
lower bound $d$ such that $d \leq i_{\alpha}$ for all 
$\alpha$.  Then Proposition~\ref{prop:image} applies, and shows that  
 $\wt{K}/K \cong \bigoplus_{i \in \mb{Z}} L[-i]^{\oplus a_i}$, where  
$a_i = 0$ for $i < d$. Also, the   $a_i$ are finite   since $\uHom_R(J,K) R$ is locally finite.
Since $\hilb L=(1-s)^{-2}$, it  is immediate that  $\hilb \wt{K}/K  =p(s)(1-s)^{-2}$.

(3) Consider the exact sequence
\begin{equation}\label{tilde2}
0 \too \uHom_R(L, \wt{K}/K) \overset{\alpha}{\too} \uExt^1_R(L, K) \overset{\beta}{\too}
 \uExt^1_R(L, \wt{K}) \overset{\gamma}{\too}   \uExt^1_R(L, \wt{K}/K)\too\cdots
\end{equation}
We will show that $\beta = 0$.  
So, suppose that  $0\not=\theta \in \uExt^1_R(L, K)_{-j}$, corresponding to
 a nonsplit extension $0 \to K \overset{\iota}{\to} N \overset{\pi}{\to} L[j] \to 0$.
  We claim that $\iota: K \to N$ 
is an essential extension.  If not, choose $0 \neq G \subseteq N$ maximal such that $\iota(K) \cap G  = 0$.   
Then there is an exact sequence $0 \to K \to N/G \to L[j]/\pi(G) \to 0$ where $K \to N/G$ is now essential.
But $\GKdim L[j]/\pi(G)\leq 1$ since $L$ is $2$-critical, and as $K$ is reflexive this forces $L[j]/\pi(G) = 0$.  
Thus $N = K \oplus G$, giving the required contradiction.  This proves the claim.  

Since $\iota$ is essential, we may take  the extension  $K \subseteq N$  inside $Q_{\gr}(R)$. 
As such,  $N \subseteq \wt{K}$ by definition  and so $\theta$ cannot induce a nontrivial extension of $\wt{K}$ by $L$.  Thus, $\beta=0$.

As $L_R$ is finitely generated, $\uHom_R(L, \bigoplus_{i \in \mb{Z}} L[-i]^{\oplus a_i}) = \bigoplus_{i \in \mb{Z}} \uHom_R(L, L[-i])^{\oplus a_i}$.   
Therefore,  since $\hilb \uHom_R(L, L) = 1/(1-s)$ by Lemma~\ref{lem:red0}, this implies that 
$\hilb \uHom_R(L, \wt{K}/K) = (\sum_{i \in \mb{Z}} a_i s^i)/(1-s)$.  
Applying   $\alpha$ shows
that $\hilb \uExt^1_R(L,K)=(\sum_{i \in \mb{Z}} a_i s^i)/(1-s)$, as well.

\medskip
$(4)$   If $\uExt^1_R(L, L) = 0$,  then  $ \uExt^1_R(L, \wt{K}/K) = 0$ in \eqref{tilde2}. Since $\beta=0$,  this forces 
$ \uExt^1_R(L, \wt{K}) =0$. 
\end{proof}
\noindent

We now come to one of the main results of the paper by showing that,  under mild conditions, in  applying the tilde operation 
 to $R$ itself one  obtains a ring $\widetilde{R}=\wt{R}_L$. As we show later in the section,   this operation is a good  non-commutative 
analogue of blowing down a line of self-intersection $-1$. In fact,  the self-intersection condition  is not quite the right concept when
 $\rqgr R$ is not smooth, and so the theorem is stated under the weaker condition \eqref{eq:line}.

\begin{theorem}\label{thm:I.1}
 Let $R$ be an elliptic algebra with $R/gR \cong B(E, \mc{M}, \tau)$.  Let $L = R/J$ be a right line module 
 with $\Div L = p$, satisfying 
\begin{equation}
\label{eq:line}
\hilb \uEnd_R(J)=\hilb  R.
\end{equation}   
Then  the module $\wt{R}=\wt{R}_L$  constructed in Lemma~\ref{lem:tilde} is a connected graded subalgebra of $R_{(g)}$.  
It  is also equal to $\uHom_R(J,J) R$ and satisfies the following properties.
\begin{enumerate}
 \item  As right $R$-modules, $\wt{R}/R \cong  \bigoplus_{i \geq 1} L[-i]$.
\item As left $R$-modules, $ \wt{R}/R \cong \bigoplus_{i \geq 1} L^{\vee} [-i]$, where  $L^{\vee}= \uExt^1_R(L, R)[1]$ is the dual line module.  
\item $\wt{R}$ is an  elliptic algebra with $\wt{R}/g\wt{R} \cong B(E, \sM(\tau^{-1}(p)), \tau)$. 
\end{enumerate}
\end{theorem}

\begin{remark}\label{blowdown}  (1) In the notation of the theorem, we  say that $\wt{R}_L$   is obtained by {\em blowing down} 
or {\em contracting} $L$ (or alternatively its dual $L^\vee$).

(2) 
Note that by Theorem~\ref{thm:one}, if $(L \bigdot L) = -1$ then \eqref{eq:line} holds and so 
Theorem~\ref{thm:I.1} gives a method of contracting a line of self-intersection $(-1)$ on an elliptic algebra.

(3) If $\rqgr R$ is smooth then the conditions $(L \bigdot L) = -1$ and \eqref{eq:line}  are equivalent (see Remark~\ref{smooth-thm:one}).  
However, when $\rqgr R$ is not smooth the later condition can definitely be weaker and   there do exist  line modules $L$ with  
$(L \bigdot L) \not= -1$ that can still be blown down by the theorem. See Corollary~\ref{eg-cor1} for one such example.

\end{remark}

\begin{proof}   Lemma~\ref{lem:Rdual}(1) implies that 
  $\hilb \uExt^1_R(L, R) = s/(1-s)^2 = (s+ s^2 + \dots)/(1-s)$.
Thus (1) holds  by comparing parts (2) and (3) of Lemma~\ref{lem:tilde}.
Also, by Lemma~\ref{lem:tilde}(1) we have  $\wt{R} = \uHom_R(J, R)R \subseteq R_{(g)}$.
Indeed,  since $ \uHom_R(J,R)$ is automatically a left $R$-module, $\wt{R}$ is actually an $R$-bimodule. Note that, to this point,  we have not used \eqref{eq:line}.

The main part of the proof will be to prove  that $\wt{R}$ is a subalgebra of $Q_{\gr}(R)$, the first step in which will be to prove that 
 $\wt{R} = \uEnd_R(J) R$. 
 
 Certainly $\uEnd_R(J) R \subseteq \uHom_R(J, R)R=\wt{R}$.
Since $J$ is $g$-divisible,    Lemma~\ref{lem:ideal}(1) implies that $\uEnd_R(J)$ is $g$-divisible.
By hypothesis,  $\overline{R} = \bigoplus_{n \geq 0} H^0(E, \mc{M}_n) t^n$ and 
$\overline{J} = \bigoplus_{n \geq 0} H^0(E, \mc{M}_n(-p)) t^n$ and so, by 
  Lemma~\ref{lem:uhomB},
  \[\uEnd_{\overline{R}}(\overline{J}) = \bigoplus_{n \geq 0} H^0(E, \mc{M}_n(-p + \tau^{-n}(p)))=B(E, \mc{M}', \tau)\qquad\text{for }\ \mc{M}' =  \mc{M}(-p + \tau^{-1}(p)).\]
Let $B = B(E, \mc{M}, \tau)$.   
Clearly $\overline{\uEnd_R(J)} \subseteq \uEnd_B(\overline{J}).$ Conversely, \eqref{eq:line} 
and $g$-divisibility imply that  
$\hilb \overline{\uEnd_R(J)}=\hilb  \overline{R} =\hilb  B(E, \mc{M}', \tau)=\hilb  \uEnd_B(\overline{J})$. Hence, 
$
\overline{\uEnd_R(J)} =  \uEnd_B(\overline{J}) = B(E, \mc{M}', \tau).
$
This in turn implies that   $\overline{\uEnd_R( J) R} =
 \overline{\uEnd_R(J)}\cdot \overline{R} = B'B,$ where $B'= B(E, \mc{M}', \tau)$.

We claim next that $B' B = B''$, where $B'' = B(E, \mc{N}, \tau)$ for $\mc{N} = \mc{M}(\tau^{-1}(p))$.
Indeed, by  Lemma~\ref{lem:basics for B} and the fact that all the relevant invertible sheaves are generated by their global sections,
\begin{gather*}
\sum_{i= 0}^n B'_i B_{n-i} \ = \  \sum_{i=0}^n H^0\left(E, \mc{M}_i\left(-p + \tau^{-i}(p)\right)\right) 
\cdot H^0\left(E, \mc{M}_{n-i}^{\tau^i}\right)  
 \ = \  \sum_{i=0}^n H^0\left(E, \mc{M}_n\left(-p + \tau^{-i}(p)\right)\right) \\
 \ = \   H^0\Big(E, \,\, \sum_{i=0}^n\mc{M}_n\left(-p + \tau^{-i}(p)\right)\Big)
 \ = \ H^0\left(E, \mc{M}_n\left(\tau^{-1}(p) + \dots + \tau^{-n}(p)\right)\right) = B''_n,
\end{gather*}
proving the claim.

Although we do not know a priori that $\uEnd_R(J) R$ is $g$-divisible,  the previous paragraph at least gives the inequality 
$\hilb \uEnd_R(J) R \geq (\hilb B'')/(1-s)$.   Since Riemann-Roch 
gives $\hilb B'' = \hilb B + \sum_{i \geq 1} s^i/(1-s)$, 
  this implies that $(\hilb B'')/(1-s) = \hilb \wt{R}$.  This forces 
$\uEnd_R(J)R = \uHom_R(J, R) R = \wt{R}$ as desired at the beginning of the proof. It follows  that
$\wt{R}/\wt{R}g = B''$ and so  $\wt{R}$ is indeed
 $g$-divisible.

Consequently,   
\[R\cdot \uEnd_R(J) \subseteq R \cdot\uHom_R(J, R)  = \uHom_R(J, R) \subseteq   \uEnd_R(J) R.\]  Thus 
$\left(\uEnd_R(J) R\right)^2 =\left(\uEnd_R(J) \right)^2 R =\uEnd_R(J)R$
and so $\wt{R}$ is indeed a subalgebra of $Q_{\gr}(R)$.  Moreover, since   $\wt{R}/\wt{R}g = B''$, 
it follows that  $\wt{R}$ is an elliptic algebra. 

Finally, all of the above arguments hold  for the 
 left line module $L^{\vee} = R/J^{\vee}$, and so 
one obtains a elliptic subalgebra $\wt{R}^{\ell}$ of $Q_{\gr}(R)$ with ${}_R (\wt{R}^{\ell}/R) \cong \bigoplus_{i \geq 1} L^{\vee}[-i]$.

By Lemma~\ref{lem:Rdual}(3),   $M = \uHom_R(J^{\vee},R)$ is a right $R$-module such that 
$M/R \cong L[-1]$. Dually, $M^{\vee} = \uHom_R(J,R)$ is a left $R$-module with 
$M^{\vee}/R \cong L^{\vee}[-1]$.
In particular, $M_1 J \subseteq R$ and so $M_1 \subseteq \uHom_R(J,R)_1 = M^{\vee}_1$.   
By symmetry,   $M_1 = M^{\vee}_1$.  Now by construction, $\wt{R}_1 = [\uHom_R(J,R)R]_1 = M^{\vee}_1$ and so, dually, $\wt{R}^{\ell}_1 = M_1$.
 Thus $\wt{R}^{\ell}_1 = \wt{R}_1$.  Since $\wt{R}$ and $\wt{R}^{\ell}$ are both generated in degree $1$ by Proposition~\ref{prop:elliptic},  
it follows that they are equal.  This proves part (2) and completes the proof.
\end{proof}

\begin{corollary}\label{cor:tilde-op}
Let $R$ be an elliptic algebra with a line module $L = R/J$ satisfying \eqref{eq:line}, and let 
$K \subseteq R_{(g)}$ be a reflexive $g$-divisible finitely generated $R$-module.  Then $\wt{K}_L$ 
is a right $\wt{R}_L$-module.
\end{corollary}

\begin{proof}
By the proof of Theorem~\ref{thm:I.1},  $\wt{R} =\wt{R}_L = \uEnd_R( J) R$, while
 $\wt{K} =\wt{K}_L= \uHom_R(J, K) R$ by Lemma~\ref{lem:tilde}.  We also saw that 
$R\cdot  \uEnd_R(J) \subseteq \uEnd_R( J) R$ in the proof of Theorem~\ref{thm:I.1}. 
Thus
\[
\wt{K} \wt{R} = \uHom_R(J, K)  R\cdot \uEnd_R(J) R \subseteq \uHom_R(J, K)
 \uEnd_R( J) R \subseteq \uHom_R(J, K) R = \wt{K}
\]
as required.
\end{proof}

To conclude this section, we explain how the above construction of ring-theoretic  blowing 
down is formally the inverse 
of noncommutative blowing up.  When $R$ is an elliptic algebra of degree  $\mu\geq3$ with 
$R/Rg = B(E, \mc{M}, \tau)$, then 
for any $p \in E$ one may define the ring-theoretic blowup of $R$ at $p$ to be the subring 
$R' = R(p) \subseteq R$  generated by $R'_1 = \{x \in R : \overline{x} \in H^0(E, \mc{M}(-p)) \}$
 (see \cite{R-Sklyanin} and \cite{RSSshort} for the basic properties of these blowups).  
Then \cite[Theorem~1.1]{RSSshort}
implies that $R'$ is also elliptic,  of degree $\mu-1$, with $R'/R'g \cong B(E, \mc{M}(-p), \tau)$.  
The ring $R'$ automatically has an \emph{exceptional}  line module $L$ satisfying
$(R/R') \cong \bigoplus_{i = 1}^{\infty} L[-i]$ as right $R'$-modules.  As we show in the next theorem, one can then recover $R$ by blowing down this line.  

As an aside, we note that one can also allow elliptic algebras to have degree $1$ or $2$ (though in this paper the  definition excludes them) and  
the ring-theoretic blowup at a point of an elliptic algebra of  degree  $\mu= 2$ can still  be defined.
  However, such a  blowup will not now be generated in degree one and so  a more complicated 
definition is necessary     (see \cite{RSSshort}
for the details).

\begin{theorem}
\label{thm:inverse} {\rm (1)}
Let $R$ be an elliptic algebra, of degree at least $4$, with $R/Rg = B(E, \mc{M}, \tau)$.  
If 
  $p \in E$, then the exceptional line module $L = R(p)/J$ 
of the ring-theoretic blowup $ R(p)$ satisfies \eqref{eq:line}; in fact, $\uEnd_{R(p)}(J) = R(\tau(p))$. 
 Thus the blowdown $\wt{R(p)}_L$ of  $R(p)$ 
along   $L$ is defined, and it  equals  $R$.

{\rm (2)} Conversely, suppose that $R'$ is an elliptic algebra  with a line module $L$   satisfying \eqref{eq:line}
 and  set   $\wt{R}'=\wt{R}'_L$.  If   $\Div L = \tau(p)$, then the ring-theoretic blowup 
$\wt{R}'(p)$   equals $R'$.
\end{theorem}

\begin{proof}  
We  remark that the apparent shift $p\mapsto \tau(p)$ between parts (1) and (2) comes from the fact that, by 
\cite[Lemma~9.1]{R-Sklyanin}, the line module $L$ in part (1) has $\Div L=\tau(p)$.

 (1)
We first claim that $R(p)_1 R_1 = R_1 R(\tau(p))_1$.  Indeed, since $g R_1 \subseteq R(p)_1 R_1$, it 
follows that \[R(p)_1 R_1 = \{ x \in R_2 : \overline{x} \in \overline{R(p)_1 R_1} = H^0(E, \mc{M}_2(-p)) \}.\] 
 A similar 
calculation shows that $R_1 R(\tau(p))_1$ is equal to the same  subspace of $R_2$, 
proving the claim.

Since $R/R(p) \cong \bigoplus_{i \geq 1} L[-i]$,  clearly   $R_{\leq 1} R(p)/R(p) \cong L[-1]$
and so    $J = \{x \in R(p) : R_1 x \subseteq R(p) \}$.  Thus 
$R_1 R(\tau(p))_1 J_1 = R(p)_1 R_1 J_1 \subseteq R(p)_1 R(p)_2 \subseteq R(p)$, and so 
$R(\tau(p))_1 J_1 \subseteq J$.   By  Lemma~\ref{lem:Rdual}(2), the line ideal  $J$ 
  is generated in degree $1$.  Thus
$R(\tau(p))_1 J_1 R(p) \subseteq J$ implies $R(\tau(p))_1 J \subseteq J$.  Since $R(\tau(p))$ is also
generated in degree $1$ it follows that $R(\tau(p)) \subseteq \uEnd_{R(p)}(J)$, 
and so $\hilb \uEnd_{R(p)}(J) \geq \hilb R(\tau(p)) = \hilb R(p)$.  Conversely, 
$\hilb \uEnd_{R(p)}(J) \leq \hilb R(p)$ follows, for example, from
  \eqref{eq:homJJ}.  Thus $\hilb \uEnd_{R(p)}(J) = \hilb R(\tau(p))$ and hence
$\uEnd_{R(p)}(J) = R(\tau(p))$.

Now Theorem~\ref{thm:I.1} applies to define the blowdown $\wt{R(p)}$ of $R(p)$ along $L$.  
By that theorem,  $\wt{R(p)} = \uEnd_{R(p)}(J) R(p) = R(\tau(p)) R(p)$.  Since 
$R_1 = R(\tau(p))_1 + R(p)_1 = \wt{R(p)}_1$ and both $R$ and $\wt{R(p)}$ are 
generated in degree $1$ as algebras, necessarily $\wt{R(p)} = R$.   

(2)   In this case, the blowdown  $\wt{R}'$   satisfies 
$\wt{R}'/\wt{R}'g = B(E, \mc{M}(p), \tau)$, 
by Theorem~\ref{thm:I.1}.  The blowup of $\wt{R}'$ at the point $p$ is thus the subring of $R'$
 generated in degree $1$ 
by $\{ x \in \wt{R}'_1 : \overline{x} \in H^0(E, \mc{M}) \}$. This is precisely $R'$.
\end{proof}

 
\section{Smoothness}\label{SMOOTH} 

A key feature of the commutative geometry described in Proposition~\ref{prop1} and Theorem~\ref{CCC} is that if $\pi: Y \to X$ is
 a birational morphism of surfaces where the exceptional locus is a $(-1)$ line, then $Y$ is nonsingular if and only if $X$ is nonsingular.
In this section we will prove a natural analogue of this result (see Theorem~\ref{thm:bigsmooth}).   
Recall that for a graded ring $R$ we say that $\rqgr R$ is \emph{smooth} if this category has finite homological dimension.  
 
\begin{theorem}\label{thm:bigsmooth}
 Let $T$ be an elliptic algebra of degree $\geq 4$ associated to the elliptic curve $E$, and let $p \in E$.
Let $L$ be  the exceptional line module for the blowup $T(p) \subseteq T$.
The following are equivalent:
\begin{enumerate}
 \item  $\rqgr T(p)$ is smooth.   
\item $\rqgr T$ is smooth and  $\pdim_{T(p)^\circ} L^\circ <\infty$.  
\end{enumerate}
\end{theorem}
 
 We remark that by Lemma~\ref{lem:pdL}, we have $\pdim_{T(p)^\circ} L^\circ <\infty$ $\iff$ $\pdim_{T(p)^\circ} L^\circ =1$.
Note also that we will show later that blowing up the point $p$ in the blowup $T(p)$ leads to a non-smooth noncommutative 
 scheme (see Corollary~\ref{eg-cor2} for the details) and so the extra conditions  of the theorem are necessary.

As an immediate corollary of Theorem~\ref{thm:bigsmooth}, we obtain: 

\begin{corollary}\label{cor:smooth}
Let $R$ be an elliptic algebra of degree $\geq 3$ and suppose that $L$ is a line module with $\pdim L^\circ < \infty$ and $(L\bigdot L) =-1$.  Let  $\wt{R}=\wt{R}_L$ be the blowdown of $R$   constructed by Theorem~\ref{thm:I.1}.

Then $ \rqgr \wt{R}$ is smooth if and only if  $\rqgr R$ is smooth.
\end{corollary}

\begin{proof}[Proof of Corollary~\ref{cor:smooth}]  
By Theorem~\ref{thm:inverse} $R = \wt{R}(q)$, where $q = \tau^{-1}( \Div L)$.  Thus the result is a direct application of  Theorem~\ref{thm:bigsmooth}.
\end{proof}

The rest of the section is devoted to  the proof of Theorem~\ref{thm:bigsmooth}. 
We work mostly in the localised category of modules over $U=T^\circ$;  note that, by Lemma~\ref{lem:Dansmooth}, $\rqgr T$ is smooth if and only if  $\gldim T^\circ < \infty$.

\begin{proposition}\label{prop:smooth}
 Let $U$ be a noetherian domain with division ring of fractions $D=Q(U)$ and  a projective right ideal $J$. Set  $ L = U/J$ and $ L^\vee = \Ext^1_U(L, U)$.   
Let $U \subset V\subset D$ be an overring  satisfying:
\begin{enumerate}
\item[(a)] $(V/U)_U \cong L^{\oplus \JJ} $ for some index set $\JJ$;  
\item[(b)] $ \Hom_U(L, V) = \Ext^1_U(L, V) =0 $ and the same for $L^\vee$.
\end{enumerate}
Then:
\begin{enumerate}
\item $L \otimes_U V = 0 = \Tor_1^U(L, V)$.

 \item 
 Let $L^\perp$ be the full subcategory of $\rMod U$ consisting of modules $M$ satisfying $\Hom_U(L, M) = \Ext^1_U(L, M) = 0$.  Then $L^\perp \simeq \rMod V$.
\item $\gldim V \leq \gldim U$.  In particular, if $\gldim U < \infty$ then $\gldim V < \infty$.
\end{enumerate}
\end{proposition}

\begin{proof}
 $(1)$ We first compute $J \otimes_U V$.  Since  $V\subset D$,
 there is an exact sequence
 \[   \Tor^U_2(L,D/V) \to \Tor^U_1(L, V) \to \Tor^U_1(L, D).\]
 Using that $\pd(L)=1$ and that $D$ is a flat $U$-module,   
the outside terms are zero  in this  sequence and so $\Tor^U_1(L, V)=0$. Thus    the natural map
  $\phi: J \otimes_U V \to U \otimes_U V  = V$ is injective
and we identify $J \otimes_UV$ with $JV=\text{Im}\,\phi$. 
  
   From the exact sequence
$0 \to U \to J^* \to L^\vee \to 0$ and (b), the natural map $\Hom_U(J^*, V) \to \Hom_U(U,V)=V$ is an isomorphism.  
In particular, the inclusion $U \subseteq V$ lifts to an inclusion
 $J^* \subseteq V$ and so, by the Dual Basis Lemma, $J V \supseteq J J^*\ni 1$.  
So $J \otimes_U V = V$.  
It follows by tensoring the exact sequence $0 \to J \to  U \to  L \to 0$ with $V$ that $L \otimes_U V = 0$.

$(2)$
Consider the functors
$F = \Hom_U(V, -):  \rMod U \to \rMod V$ and $G = \operatorname{res}_U:  \rMod V \to \rMod U$.  
We claim that $F$ and $G$ give inverse equivalences between $L^\perp$ and $\rMod V$.

Let $M \in L^\perp$, and consider the exact sequence
\[   \Hom_U(L^{\oplus \JJ}, M) \to  \Hom_U(V, M) \to  \Hom_U(U, M) \to \Ext^1_U(L^{\oplus \JJ}, M) \]
induced from (a).  
As the outside terms are zero, this provides a natural isomorphism of $U$-modules: $GF(M)  = \Hom_U(V,M) \ \buildrel{\sim}\over\too\ M$. Hence 
$M$ carries a natural right $V$-action. Similarly, if $M' \in L^\perp$ as well then a $U$-module homomorphism $\theta: M\to M'$ induces a $V$-module homomorphism $ \Hom_U(V, M)\to  \Hom_U(V, M')$, from which it follows that $\theta$ is  already a $V$-module map. 
Thus $GF \cong \id_{L^\perp}$.

Now let $N \in \rMod V$.
From the spectral sequence $\Ext^p_V(\Tor_q^U(L,V),N) \Rightarrow \Ext^{p+q}_U(L,N)$ \cite[Theorem~10.74]{Rotman}, we have 
$\Hom_U(L,N) = 0 =   \Ext^1_U(L, N)$, and  so $G(N)\in L^\perp$.
 There is a natural $V$-module map $N \to \Hom_U(V, N)$ given by $n \mapsto (s \mapsto ns)$, which is the inverse of the natural isomorphism  
 $\Hom_U(V, N) \to N$  discussed above and so it is also an $V$-module isomorphism.   Thus $FG \cong \id_{\rMod V}$
 on objects.   It is routine to check that this respects morphisms   and  so $F,G$ are indeed inverse equivalences.

$(3)$
Let $M, N \in \rMod V$, which we identify with $L^\perp$, using $(2)$.  
It is clearly sufficient to prove that
\beq\label{sameext}
\Ext^i_V(M, N) \cong \Ext^i_U(M, N)\qquad \text{for}\ i\geq 1.
\eeq
To prove this we will use the spectral sequence $(2)_{4}$ from \cite[Section~XVI.5, p349]{CE} for the injection $\phi: U\to V$. We begin with a couple of observations.
 
By part (2) and the fact that $L_U$ has projective  dimension $\pd_U(L)\leq 1$, we have  $ \Ext^j_U(L,N)=0$ for all $j\geq 0$.
Now consider the long exact sequence obtained by applying $\Hom_U(-, N)$ to  the exact sequence $0\to U\to V \to L^{\oplus \mathbb{J}}\to 0$ 
arising from  (a).   Then certainly  $\Ext^q_U(V, N) = 0$ for $q \geq 1$. Moreover,
    $\Hom_U(V,N)=\Hom_U(U,N)=N$; thus ${}^{(\phi)}N=N$ in the notation of \cite{CE}. 
Therefore, as explained in \cite{CE},   the  cited spectral sequence collapses and  the edge homomorphism  $(3)_{4}$ from 
  \cite[Section~XVI.5, Case 4, p.349]{CE}  becomes the desired  isomorphism
$\Ext^i_V(M, N) =\Ext^i_V(M,\, {}^{(\phi)}N) \cong \Ext^i_U(M,N).$
\end{proof}

We next prove a partial converse to Proposition~\ref{prop:smooth}(3), for which we need the following result on universal extensions.

\begin{lemma}\label{lem:univ2}
Let $U$ be a noetherian $\kk$-algebra and let $L$ be a finitely generated right $U$-module satisfying $\End_U(L)= \kk$ and $\Ext^1_U(L,L)=0$.  
For any right $U$-module $Q$ such that $\Hom_U(L, Q) =0$, there is a short exact sequence
\[ 0 \to Q \to N \to \Ext^1_U(L,Q) \otimes_{\kk} L \to 0,\]
for some $N \in L^\perp$.
\end{lemma}
\begin{proof}
Let $E = \Ext^1_U(L, Q)$ and  choose a basis $\{e_i\}_{i\in \mb I}$ for $E$ as a $\kk$-vector space.  
As in \cite[Lemma~4.2]{EH2008},  construct a short exact sequence $0 \to Q \to N \to E\otimes_{\kk} L \to 0$ such that
 the pullback under $e_i \otimes \id_L : L \to E \otimes_{\kk} L$  is the extension  $0 \to Q \to N_i \to L \to 0$  corresponding to $e_i$.
(It can be shown by a diagram chase that $N$ is the element of $\Ext^1_U(E \otimes_{\kk} L ,Q) $ corresponding to $\id_E$ via the natural
 isomorphism  $\alpha: \Hom_{\kk}(E,E) \to \Ext^1_U(E \otimes_{\kk} L, Q)$.) 
By construction,  the diagram
\[ \xymatrix{0 \ar[r] & Q \ar[r] \ar@{=}[d] & N_i \ar[d] \ar[r] & L \ar[d]^{e_i \otimes \id_L} \ar[r] & 0 \\
 0 \ar[r] & Q \ar[r] & N \ar[r] & E\otimes_{\kk} L \ar[r] & 0 }\]
 commutes.
 Applying $\Hom_U(L, -)$, we obtain a commutative diagram:
 \beq \label{dagdiag}
 \xymatrix{0 \ar[r] & \Hom_U(L,Q ) \ar[r] \ar@{=}[d] & \Hom_U(L,N_i) \ar[d] \ar[r] & \Hom_U(L,L) \ar[d]^{(e_i \otimes \id_L) \circ (-)} \ar[r]^(0.6){\delta_i} &E \ar@{=}[d] \\
 0 \ar[r] & \Hom_U(L,Q ) \ar[r] & \Hom_U(L,N) \ar[r] & \Hom_U(L, E\otimes_{\kk} L) \ar[r]_(0.7){\delta} & E. }
 \eeq
 Thus
 $\delta(e_i \otimes \id_L) = \delta((e_i \otimes \id_L) \circ \id_L) = \delta_i(\id_L)$, and this equals  $e_i$ from 
the standard way that elements of $\Ext^1_U(L, Q)$ correspond to extensions.  Since $L$ is finitely generated and $\End_U(L) \cong \kk$, it follows  that $\delta$ is an isomorphism.  
 
 We now use that $\Hom_U(L, Q) = 0=\Ext^1_U(L,L) $. 
 Extending the bottom row of \eqref{dagdiag} gives the  long exact sequence
 \[
0\to \Hom_U(L,N)\to \Hom_U(L, E\otimes_{\kk} L)\stackrel{\delta}{\to}  E \to \Ext^1_U(L, N) \to  \Ext^1_U(L, E \otimes_{\kk} L).
\]
But $\Ext^1_U(L, E\otimes_{\kk} L) \cong E \otimes_{\kk} \Ext^1_U(L,L) = 0$.  
Since $\delta$ is an isomorphism, it follows that $N \in L^\perp$.
\end{proof}
  
\medskip

\begin{proposition}~\label{prop:smoothconv}
 Suppose that $U$, $V$ and $L$ satisfy the hypotheses of  Proposition~\ref{prop:smooth} and, in addition, that $\Hom_{U}(L,L) = \kk$ and $\injdim U=d < \infty$.
If $\gldim V < \infty$ then $\gldim U < \infty$.
\end{proposition}

\begin{proof} Let $M,M'\in \rMod U$. We need to prove that $\Ext^i_U(M,M')=0$ for $i\gg 0$. 
Take exact sequences $0\to Q\to F\to M\to 0$ and  $0\to Q'\to F'\to M'\to 0$
for free modules $F,F'$ and  
consider the induced long exact sequences of $\Ext$ groups. Using that $\injdim F=\injdim U =d$,
 it follows that $\Ext^i_U(M,M')  \cong\Ext^{i+1}_U(M,Q') \cong \Ext^i_U(Q,Q')$ for $i>d$. 
 Thus, it  suffices to prove that $\Ext^j_U(Q,Q')=0$ for $j\gg 0$.  Since $Q$ and $Q'$ are Goldie torsionfree as defined on 
 page \pageref{tors-defn}, Lemma~\ref{lem:univ2} applies and produces exact sequences 
 $0\to Q\to N\to E\otimes_\kk L\to 0$ and 
 $0\to Q'\to N'\to E'\otimes_\kk L\to 0$, where $N,N'\in L^\perp$.

Once again, from the  induced long exact sequences for $\Ext$ groups, it suffices to prove that 
\begin{equation}\label{conv1}
\Ext^k_U(H,H')=0 \quad \text{ for } \quad k\gg 0,  \quad \text{where}\
H\in\{L,N\}   \  \text{and}\ H'\in \{L,N'\}. \end{equation}
Since $L=U/J$, with $J$ projective, $\injdim(L)<\infty$ and $\pd(L)=1$.  So certainly \eqref{conv1} holds 
if either $H=L$ or $H'=L$. The remaining case, where $H$ and $H'\in L^\perp$, follows from \eqref{sameext}.
\end{proof}

We now prove Theorem~\ref{thm:bigsmooth}.
\begin{proof}[Proof of Theorem~\ref{thm:bigsmooth}]  

Suppose first that $\rqgr T(p)$ is smooth.   Recall that, by Lemma~\ref{lem:Dansmooth}, if $R$ is an elliptic algebra then 
$\rqgr R$ is smooth if and only if $\gldim R^{\circ}<\infty$.   Thus $\gldim T(p)^{\circ} <  \infty$, and in particular 
$\pdim_{T(p)^{\circ}} L^{\circ} < \infty$  (and thus $\pdim_{T(p)^{\circ}} L^{\circ} = 1$ by 
Lemma~\ref{lem:pdL}). 
  By Theorem~\ref{thm:inverse}, $T$ is the blowdown of $T(p)$ along $L$, and $J$ satisfies the condition $\hilb \uEnd_{T(p)}(J) = \hilb T(p)$ 
  as in \eqref{eq:line}.   Since $L^{\circ} = T(p)^{\circ}/J^{\circ}$, the right ideal $J^{\circ}$ is projective, 
and  by Theorem~\ref{thm:one} it follows that $\uExt_{T(p)}^1(L, L) = 0$.  By Lemma~\ref{lem:tilde}(3), $\uExt^1_{T(p)}(L, T) = 0$.
Since by Lemma~\ref{lem:Rdual}(4) $\uExt^1_{T(p)}(L^\vee, L^\vee) = 0$, applying Lemma~\ref{lem:tilde}(3) on the left gives that 
$\uExt^1_{T(p)}(L^\vee, T) =0$.  Finally, it is obvious that $\uHom_{T(p)}(L, T) = 0 = \uHom_{T(p)}(L^{\vee}, T)$.  
Now using Lemma~\ref{lem0.1}, it follows from the above observations that the hypotheses of Proposition~\ref{prop:smooth}
hold with $V = T^{\circ}$, $U = T(p)^{\circ}$.  Thus $\gldim U < \infty$, Proposition~\ref{prop:smooth} implies 
$\gldim V < \infty$, and so $\rqgr T$ is smooth.

Conversely, suppose that $\rqgr T$ is smooth and that $\pdim_{T(p)^{\circ}} L^{\circ} < \infty$, so again  $\pdim_{T(p)^{\circ}} L^{\circ} = 1$.  
All of the arguments in the previous paragraph then go through to show that the hypotheses of Proposition~\ref{prop:smooth}
hold with $V = T^{\circ}$, $U = T(p)^{\circ}$.  We also have $\uHom_{T(p)}(L,L) = k[g]$ by \ref{lem:red0}(3), and $\injdim T(p) < \infty$ is part 
of the Auslander-Gorenstein condition which holds by Proposition~\ref{prop:elliptic}.  Thus applying Lemma~\ref{lem0.1}, the hypotheses of 
Proposition~\ref{prop:smoothconv} hold.  Thus $\gldim V < \infty$ implies $\gldim U < \infty$, and so $\rqgr T(p)$ is smooth.
\end{proof}

We do not know how to characterise when $\pdim L^\circ= 1$ (equivalently, $\pdim L^\circ < \infty$), although we conjecture:

\begin{conjecture}\label{conj:sierra}
 Let $T$ be an elliptic algebra with $\deg T \geq 4$ and let $T(p)$ be the blowup of $T $ at $p \in E$ with   exceptional line $L$. 
If there is no $T$-line module $L'$ with $\Div L' = \tau(p)$, then $\pdim_{T(p)^\circ}L^\circ =1$.
\end{conjecture}


 \section{An example of  undefined self-intersection}\label{EXAMPLE}

 In this section we describe an elliptic algebra $R$ with an exceptional line module $L$ for  which the self-intersection $(L\dotms L)$
  is undefined. Moreover, $(L\bigdot L) \not=-1$ 
 yet  the associated line ideal $J$ does satisfy $\hilb \uEnd_R(J)=\hilb R$. Thus one can still use Theorem~\ref{thm:I.1} to blow down
  the line $L$. This  justifies the comments made
in Remark~\ref{blowdown}  concerning that theorem and also  gives the example  
 promised before  Definition~\ref{our-intersection}.

 In a way that will shortly be made precise, the ring $R$ is obtained by blowing up  the same  point $p\in E$  twice  in the Veronese
  ring $T=S^{(3)} $ of the Sklyanin algebra $S$.  The key property, here,  is that the resulting scheme $\rqgr R$ is not smooth.  
This   may be explained by analogy with the commutative situation: iterated ring-theoretic blowups of $T$ are analogs of commutative rings of the form 
$A = \bigoplus_{n \geq 0} H^0(\mb{P}^2, (\mc{I} \otimes \mc{O}(9))^{\otimes n})$, where $\mc{I}$ is an ideal sheaf defining a zero-dimensional
 subscheme $Z$ of $\mb{P}^2$.  When $Z$ is not reduced, a failure of ampleness means that $\Proj A$ need not be isomorphic to  the blow-up
  of $\mathbb{P}^2$ at $Z$ and, moreover, $\Proj A$ need  not be smooth.
Of course,  this also shows that the analogy between commutative and noncommutative blowups is less precise in the non-generic situation.

 \begin{notation}\label{starting}
 We start with the relevant notation, which will be fixed throughout the section. Let $T=S^{(3)}$ be a Sklyanin elliptic algebra, as 
 defined in Example~\ref{sklyanin-defn} for an automorphism $\sigma$ of infinite order, with quotient ring $T/gT=B=B(E,\mathcal{M},\tau)$. 
 Fix a point $p\in E$. Following \cite{R-Sklyanin} we blow up $p$ once to get a ring $R'=T(p)$ and then blow up $R'$ at $p$ again  to give the ring $R=R'(p)=T(2p)$.  
 
  Let $L=R/J$ be the exceptional $R$-line module,  with line ideal $J$,  corresponding to the extension $R\subset R'$; this exists by
   \cite[Lemma~9.1]{R-Sklyanin}. Similarly, let $L'$ be the exceptional $R'$-line module corresponding to the extension $R'= T(p)\subset T$.  Finally, write  
  $$X = R'_{\leq 1} R  \ \subseteq \ Y = T_{\leq 1} R\ \subseteq\  Z = T_{\leq 1} R'\ \subseteq\  T.$$   \end{notation}

The following fact, due to Simon Crawford,  will be used several times.

\begin{proposition}[\cite{SC}]
\label{prop:SC}
The localised algebra $R^\circ=R[g^{-1}]_0$ is simple.  
\end{proposition}

 We note the following useful facts about the line ideals  $L$ and $L'$.
 
 \begin{lemma}\label{eg-lem-a}  
(1)  The line module $L$ has no proper $g$-torsionfree factor $R$-modules; equivalently, $L^\circ$ is a simple $R^\circ$-module.

(2) As $R$-modules, $  Z/R' = (Y+R')/R'\cong L[-1]$  and hence $L \cong L'$.  
 \end{lemma}
 \begin{proof}
(1) By Proposition~\ref{prop:SC},  $R^\circ$ has no finite-dimensional modules.  
The simplicity of $L^\circ$ is then a consequence of the 1-criticality of $L^\circ$, which follows from Lemma~\ref{lem:basics}.
 
(2) We first make some calculations in the Sklyanin algebra $S$.
  Recall that $S_1$ may be identified with $H^0(E, \sL)$ for some invertible sheaf $\sL$ on $E$ of degree 3.
  For $r \in E$, let $W(r) = H^0(E, \sL(-r)) \subset S_1$. Then
   $S_1 W(r) = W(\sigma^{-1}(r))S_1$, by \cite[Lemma~4.1]{R-Sklyanin}, while 
  $R'_1 = W(p) S_2$ by \cite[Lemma~4.2]{R-Sklyanin}.  
We also have $R_1 = W(p) W(\sigma (p)) S_1$ by \cite[Lemma~4.6]{R-Sklyanin}.

Let $V =  W(\sigma^3(p)) W(\sigma (p)) W (\sigma^2 (p))$.  We show next that $J_1 = V$.  
We calculate that 
\[
R'_1 V = W(p) S_2 W(\sigma^3(p)) W(\sigma(p)) W (\sigma^2(p)) = W(p) W(\sigma(p)) S_1 W(p) W(\sigma(p)) S_1 = R_2.
\]
We have $X/R = R'_{\leq 1}R/R \cong L[-1] \cong (R/J)[-1]$, as a consequence of Theorem~\ref{thm:I.1}, since $R'$ is the blowdown of $R$ along $L$ 
by Theorem~\ref{thm:inverse}.  Thus $J = \{ x \in R : R'_1 x \subseteq R \}$ and so $V \subseteq J_1$.    
The Hilbert series of $J$ is known, and $\dim_\kk J_1 = 6$.  On the other hand, using Lemma~\ref{lem:basics for B} we calculate 
in $B(E, \mc{M}, \tau)$ that $\dim_\kk \overline{V} = 6$.  Thus $V = J_1$.

We now claim that 
\beq\label{oreo} T_1 J \subseteq X.\eeq 
This follows from the calculation  
  \beq\label{oreo2}
  T_1 J_1 = S_3 W(\sigma^3(p)) W(\sigma (p)) W (\sigma^2 (p)) = W(p) S_2 W(p) W(\sigma(p))S_1 = R'_1 R_1 = X_2
  \eeq
and the fact that by Lemma~\ref{lem:Rdual}(2), $J$ is generated in degree one as a right $R$-ideal.

Now since $T$ is the blowdown of $R'$ along $L' = R'/J'$, similarly to the  above we obtain $Z/R' = (T_{\leq 1} R')/R' \cong L'[-1]$ and $J' = \{ x \in R' : T_1 x \subseteq R' \}$.
Since $X \subseteq R'$, \eqref{oreo} gives $J \subseteq J'$.  By construction, $\dim_\kk T_1/R_1'=1$, so write  $T_1 = R_1' \oplus \kk a$ for some $a\in T_1$.  
As $J \subseteq J'$, there is a nonzero homomorphism $\pi: L[-1] \cong (R/J)[-1] \to Z/R' \cong L'[-1]$ sending $1$ to $a$.  
Since $L^\circ$ is simple,  $L$ has no  proper $g$-torsionfree  factors, and so $\pi$ is injective. By comparing Hilbert series, it is an isomorphism.
\end{proof}

 We further have:
 \begin{lemma}\label{eg-lem-b}
As $R^\circ$-modules, $Y^\circ$ is projective, while $J^\circ$ is not.
\end{lemma}
\begin{proof}
Let $	q = \tau (p)$.  
 By Theorem~\ref{thm:inverse}, $\uEnd_R(J)\cong F=R'(q)$. We now pass to the ring $R^\circ $ and notice that, by standard 
 localisation theory,  $F^\circ=\End_{R^\circ}(J^\circ)$.  Moreover, $F = T(p + \tau(p))$, the blowup of $T$ at 
two consecutive points on a $\tau$-orbit, is shown in \cite[Proposition~11.2(1)]{R-Sklyanin} to have a proper ideal $I$ 
such that $F/I$ is $g$-torsionfree.  Thus $I^{\circ}$ is a proper ideal of $F^{\circ}$ and $F^{\circ}$ is not simple, whereas by Proposition~\ref{prop:SC} $R^\circ$ is simple.  
In particular,  $R^{\circ} J^{\circ} = R^{\circ}$ and thus $J^\circ$ is an $R^\circ$-generator; since $F^\circ$ is not Morita equivalent 
to $R^\circ$, it follows that $J^\circ$ is not projective as a right $R^\circ$ module. 

Now let $\ell=\tau^{-1}(p)$.  We claim that $\uEnd_R(Y) = T(2\ell )$. This will complete the proof of the lemma since 
now  $\End_{R^\circ} (Y^\circ) \cong T(2\ell)^\circ$, which is again simple by Proposition~\ref{prop:SC}.  
By  the Dual Basis Lemma, $Y^\circ$ is  therefore projective as a right $R^\circ$-module.

In order to prove the claim, we note that, from the formul\ae\ from \cite[Lemmas~4.1 and~4.6]{R-Sklyanin}
noted in the proof of Lemma~\ref{eg-lem-a},
\begin{equation}
\label{eq: pointshift}
T(2\ell)_1 T_1  \ = \ W({\sigma^{-3}} (p))W({\sigma^{-2}} (p))  S_1S_3 
\ = \ S_3 W(p)W(\sigma (p) ) S_1\  =  \ T_1R_1.
\end{equation}
Moreover, as $T(2\ell)\subseteq T$, certainly $T(2\ell)_1R\subseteq  T_1R$ and hence $T(2\ell)_1Y = T(2\ell) (R+T_1R) \subseteq Y$. 
Since $T(2\ell)$ is generated in degree one by  definition, it follows that $T(2\ell)\subseteq \uEnd_R(Y)$. 
Equation~\eqref{eq: pointshift} also implies by induction that $T(2\ell)_n T_1 = T_1 R_n$ for all $n \geq 0$, so $T(2\ell) T_1 = T_1 R$. 
It follows that $T_1 R$ is a finitely generated left $T(2\ell)$-module.   In particular, writing a $\kk$-basis $\{ x_i \}$ of $T_1$ as fractions
$x_i = y_i z^{-1}$ with a common denominator, where $y_i, z \in T(2\ell)$, we see that $T_1 R z \subseteq T(2\ell)$, 
and then $Y z = (k + T_1 R)z \subseteq T(2\ell)$.
Thus $\uEnd_R(Y) Y z \subseteq Y z \subseteq T(2 \ell)$,  which means that $\uEnd_R(Y)$ and $T(2 \ell)$ are equivalent
 orders.  Since  $T(2\ell)$ is a maximal order by  \cite[Theorem~1.1(2)]{R-Sklyanin}, the inclusion $T(2\ell)\subseteq  \uEnd_R(Y)$ is actually an equality. 
\end{proof}

We next show that $(L\bigdot L) \neq -1$.  In fact, we prove:
\begin{lemma}\label{eg-lem-c}
There is a nonsplit exact sequence
\beq\label{dagdag} 0 \to L[-1] \to Y/R \to L[-1] \to 0 .\eeq
\end{lemma}

\begin{proof} 
By \eqref{oreo}, $T_1 J \subseteq X$.   Thus there is a homomorphism $\pi: (R/J)[-1] \to Y/X = T_1 R/R'_1 R$ which sends
 $1$ to $a$, where  $T_1 = R'_1 \oplus \kk a$.  Since $T_1 R = R'_1 R + aR$, $\pi$ is surjective.   Now since $R'/X \cong \bigoplus_{i \geq 2} L[-i]$ 
as right $R$-modules by Theorem~\ref{thm:I.1}, $R'/X$ is $g$-torsionfree.  Since $R'$ is $g$-divisible, $Q_{gr}(T)/R'$ is $g$-torsionfree, 
and so $Q_{gr}(T)/X$ and thus $Y/X$ are also $g$-torsionfree.  As noted in the proof of Lemma~\ref{eg-lem-a}, $L$ has no proper $g$-torsionfree 
factor modules, and this forces $\pi$ to be injective as well.  Thus $Y/X \cong L[-1]$ as right $R$-modules.  
 
We saw that $X/R \cong L[-1]$ in the proof of Lemma~\ref{eg-lem-a}, and so the exact sequence \eqref{dagdag} exists 
as claimed.
Finally, localising \eqref{dagdag} gives the exact sequence $0\to L^\circ\to  (Y/R)^\circ \to L^\circ \to 0$. By Lemma~\ref{eg-lem-b} 
  $\pdim L^\circ > 1 = \pdim (Y^\circ/R^\circ)$. Thus, neither this sequence nor  \eqref{dagdag} is  split.
\end{proof}

\begin{corollary}\label{eg-cor1} Let $R=T(2p)$ with exceptional line module $L$, as above. Then $(L\bigdot L)\not=-1$. 
On the other hand, 
$\hilb \uEnd_R(J)=\hilb R$ and so, by Theorem~\ref{thm:I.1}, one can blow down $L$.
\end{corollary}

\begin{proof}  By Lemmas~\ref{lem:red0} and~\ref{lem0.1},  $\End_{ R^\circ}(L^\circ)=\kk$.
 On the other hand,  $\Ext_{R^\circ}^1(L^\circ,L^\circ) \neq 0$ by Lemma~\ref{eg-lem-c}. 
 Thus, by Lemma~\ref{lem0.1}, $(L\bigdot L)\geq 0>-1$. 
 Finally, by  Theorem~\ref{thm:inverse}, $\uEnd_R(J)=R'(q)$, and so the equality $\hilb \uEnd_R(J)=\hilb R$ follows from \cite[Theorem~1.1(1)]{R-Sklyanin}.
\end{proof}

The next result shows that there is a particularly interesting self-extension of $J^\circ$.

\begin{proposition}\label{eg-prop}
There is a nonsplit exact sequence
\beq\label{orange} 0 \to J^\circ \to P \to J^\circ \to 0\eeq
of $R^\circ$-modules, where $P$ is projective.
\end{proposition}

\begin{proof} Recall  that $X^\circ/R^\circ \cong Y^\circ/X^\circ \cong R^\circ/J^\circ \cong L^\circ$, from the proof of
Lemma~\ref{eg-lem-c}.
Thus the localisation of \eqref{dagdag} can be written  as:
\beq \label{dagdagprime}
0 \to X^\circ/R^\circ \to Y^\circ/R^\circ \to Y^\circ/X^\circ \to 0.
\eeq

The natural surjection $\phi':  R^\circ \to R^\circ/J^\circ \cong Y^\circ/X^\circ$ lifts to a homomorphism 
$\phi:  R^\circ \to Y^\circ/R^\circ$, which must be surjective as    $L^\circ$ is simple.
As was shown
 in the proof of Lemma~\ref{eg-lem-c},  \eqref{dagdagprime} is nonsplit.  
Let $K = \ker \phi$; thus  $K$ is projective since $\pdim(Y^\circ/R^\circ) = 1$.

Clearly $K \subseteq J^\circ=\ker (\phi')$ and  $J^\circ/K \cong L^\circ$.  This isomorphism  lifts to a map 
$\theta':  R^\circ \to J^\circ$ so that $\theta'(R^\circ)+ K = J^\circ$. This induces  a surjective homomorphism $\theta:  R^\circ \oplus K \to J^\circ$.  
It is routine to check that $\ker \theta = \{ (r, k)  \in R^{\circ} \oplus K : \theta'(r) = k \}$ and that as an $R^{\circ}$-module this is 
isomorphic to $(\theta')^{-1}(K) = J^{\circ}$.  Thus we have constructed the sequence \eqref{orange} with 
 $P = R^\circ \oplus K$. As $\pdim J^\circ > \pdim P$, it does not split.
\end{proof}

We now examine the higher Ext groups from $L^\circ$ to itself; the ultimate aim being to show that $(L\dotms L)$ is undefined.

 \begin{lemma}\label{eg-lemma3} Keep the above notation. Then:
 \begin{enumerate}
 \item $\Ext_{{R^\circ}}^n(J^\circ,J^\circ) \cong \Ext_{{R^\circ}}^{n+1}(J^\circ,J^\circ) \not=0$ for $n\geq 1$.
 
 \item $\Ext_{{R^\circ}}^n(J^\circ,J^\circ) \cong \Ext_{{R^\circ}}^{n-1}(J^\circ,L^\circ)\cong \Ext_{{R^\circ}}^n(L^\circ,L^\circ)$  for all $n\geq 2$.
 
 \item In particular, $\Ext_{{R^\circ}}^n(L^\circ,L^\circ) \cong \Ext_{{R^\circ}}^{n+1}(L^\circ,L^\circ)\not=0 $ for all $n\geq 2$.
 \end{enumerate}
 \end{lemma}
 
 \begin{proof}
 
 (1) Applying $\Hom_{R^\circ}(-, J^\circ) $ to   \eqref{orange} gives the exact sequence
 \begin{equation}\label{eg-lemma4}
 \Ext^m_{R^\circ}(P,J^\circ)\to \Ext^m_{R^\circ}(J^\circ,J^\circ)\to 
 \Ext^{m+1}_{R^\circ}(J^\circ,J^\circ)\to \Ext^{m+1}_{R^\circ}(P,J^\circ)
 \end{equation}
 for $m\geq 1$.
As $P$ is projective,
 it follows that  $  \Ext^m_{R^\circ}(J^\circ,J^\circ)\cong \Ext^{m+1}_{R^\circ}(J^\circ,J^\circ)$ for   $m\geq 1$.
 Moreover, \eqref{orange} ensures that $  \Ext^1_{R^\circ}(J^\circ,J^\circ)\not=0$ and hence $  \Ext^m_{R^\circ}(J^\circ,J^\circ)\not=0$  for  $m\geq 1$.

 (2,3) Applying $\Hom_{R^\circ}(J^\circ,-)$ to $0\to J^\circ\to {R^\circ}\to L^\circ\to 0$ gives
 $$ \Ext^m_{R^\circ}(J^\circ,{R^\circ})\to \Ext^m_{R^\circ}(J^\circ,L^\circ)\to \Ext^{m+1}_{R^\circ}(J^\circ,J^\circ)\to
  \Ext^{m+1}_{R^\circ}( J^\circ,{R^\circ} )\quad \text{for all } m\geq 1.$$
Now $J^\circ$ is CM by Lemma~\ref{lem:Rdual}(2)  and so, as $m\geq 1$,   the outside terms are zero in this equation. Hence, 
  \begin{equation}\label{eg-lemma5}
   \Ext^m_{R^\circ}(J^\circ,L^\circ)\cong \Ext^{m+1}_{R^\circ}(J^\circ,J^\circ)\cong \Ext^{m+2}_{R^\circ}(J^\circ,J^\circ)\cong  \Ext^{m+1}_{R^\circ}(J^\circ,L^\circ)
\qquad\text{ for all $m\geq 1$.}   \end{equation}
 By part (1) these groups are also non-zero.
 
From the   exact sequence $0\to J^\circ \to {R^\circ} \to L^\circ \to 0$,  one also obtains    $\Ext^m_{R^\circ}(J^\circ,L^\circ)\cong
 \Ext^{m+1}_{R^\circ}(L^\circ,L^\circ)$ for  $m\geq 1$. Combined with  \eqref{eg-lemma5} this implies that 
 $$\Ext^s_{R^\circ}(L^\circ, L^\circ) \cong \Ext^{s-1}_{R^\circ}(J^\circ, L^\circ)\cong \Ext^1_{R^\circ}(J^\circ, L^\circ)\cong
  \Ext^2_{R^\circ}(L^\circ, L^\circ) \quad \text{for all } s\geq 2.$$
Finally, by  \eqref{eg-lemma5}  and part (1),    $\Ext^s_{R^\circ}(L^\circ, L^\circ)\cong  \Ext^{s-1}_{R^\circ}(J^\circ, L^\circ)\cong 
  \Ext^{s}_{R^\circ}(J^\circ, J^\circ)\not=0,$ for all $s\geq 2$. \end{proof}

Finally, by combining  Lemma~\ref{eg-lemma3}  with Proposition~\ref{prop:0.6},  we get the promised example of an undefined self-intersection.

\begin{corollary}\label{eg-cor2}
Let $R=T(2p)$ as above, with exceptional line module $L$. Then  the self-intersection $(L\bullet_{MS}L)$ is an infinite sum 
and hence is undefined.  Further, $\gldim R^\circ = \infty$ and so $\rqgr R$ is not smooth.\qed
 \end{corollary}

   
 \section*{Index of Notation}\label{index}
\begin{multicols}{2}
{\small  \baselineskip 14pt

  Auslander Gorenstein and CM conditions    \hfill\pageref{def:gor}
   
Blowing down a line  \hfill\pageref{blowdown}

connected graded (cg) algebra  \hfill\pageref{cg-defn}

$\Div M$, the divisor of $M\in\rgr R$  \hfill\pageref{divisor-defn}

$d$-pure and $d$-critical modules  \hfill\pageref{pure-defn}

$E = \hilb \uExt^1_B(M_p, M_{p'})$,  $C = \hilb \im \delta$ \hfill\pageref{HEC-defn}

elliptic algebra,  degree of an elliptic algebra \hfill\pageref{elliptic-defn}

$\epsilon_{p,p'}$  \hfill\pageref{epsilon}

$g$-divisible \hfill\pageref{g-div}

 $g$-torsionfree modules   \hfill\pageref{div-defn}

Goldie torsionfree   modules   \hfill\pageref{tors-defn}

 $\grk M$, torsionfree rank of a $\kk[g]$-module  \hfill\pageref{grk-defn}

$H = \hilb \uHom_B(M_p, M_{p'})$    \hfill\pageref{HEC-defn}

 $\hilb M=h_M(s)$, the Hilbert series of $M$  \hfill\pageref{hilb-defn}

 $\uHom, \uExt$  \hfill\pageref{shift-defn}

  Intersection number $(M\bigdot N)$  \hfill\pageref{our-intersection}
 
 Intersection number $(M\dotms N)$  \hfill\pageref{intersect-ms}

$\wt{K}$,  extension of $K$ by shifts of a  line module  \hfill\pageref{wt-defn}

$L^\vee=\uExt^1_R(L,R)[1]$ dual line module \hfill\pageref{lem:Rdual}

 Line ideal $J$, line module $L=R/J$   \hfill\pageref{line-defn}

  Point  module $M_p$, for $p\in E$, point ideal   \hfill\pageref{point-defn}
 
 projective dimension  $\pd_R(L)$  \hfill\pageref{pd-defn}

$\rqgr R$,  quotient category of   $\rgr R$  \hfill\pageref{elementary-shift}

 $Q_{gr}(R)$,  graded quotient ring   \hfill\pageref{quot-defn}

$R^\circ=R[g^{-1}]_0$,   localisation of $R$    \hfill\pageref{grk-defn}
 
Saturated module \hfill\pageref{section-functor}
 
Shift $M[n]$ \hfill\pageref{shift-defn}

 Sklyanin algebras, $S$   \hfill\pageref{sklyanin-defn}

Smooth noncommutative scheme  \hfill\pageref{our-intersection}

 $\tau$, automorphism defining $R$ \hfill\pageref{elliptic-defn}

$R_{(g)}$, graded localisation \hfill\pageref{div-defn}

torsion and torsionfree   modules   \hfill\pageref{tors-defn}

TCR,   twisted coordinate ring $B(X,\mathcal{M},\theta)$  \hfill\pageref{TCR-defn}

  $X(L,L') =\hilb \uExt^1_R(L,L') -\hilb \uHom_R(L,L') $    \hfill\pageref{XLL-defn}

  $\omega: \rQgr R\to \rGr R$, section functor  \hfill\pageref{section-functor}

} \end{multicols}

\bibliographystyle{amsalpha}


 \def\cprime{$'$}
\providecommand{\bysame}{\leavevmode\hbox to3em{\hrulefill}\thinspace}
\providecommand{\MR}{\relax\ifhmode\unskip\space\fi MR }
\providecommand{\MRhref}[2]{%
  \href{http://www.ams.org/mathscinet-getitem?mr=#1}{#2}
}
\providecommand{\href}[2]{#2}


\end{document}